\documentclass[a4paper]{amsart}
\author{Dilip Raghavan}
\thanks{First author partially supported by National University of Singapore research grant number R-146-000-211-112.}
\address{Department of Mathematics\\
National University of Singapore\\
Singapore 119076.}
\email{raghavan@math.nus.edu.sg}
\urladdr{http://www.math.toronto.edu/raghavan}
\author{Stevo Todorcevic}
\thanks{Second author  partially supported by grants of  NSERC and CNRS}
\address{Dpartment of Mathematics, University of Toronto, Toronto, Canada, M5S 2E4.}
\address{Institut de Math\'{e}matique de Jussieu, UMR 7586, Case 247, 4 place Jussieu, 75252 Paris Cedex, France.}
\email{stevo@math.toronto.edu, todorcevic@math.jussieu.fr}
\date{\today}
\title{Suslin trees, the bounding number, and partition relations}
\usepackage{amssymb, amsmath, amsthm, mathrsfs, enumerate, amsfonts, latexsym,
bbm, mathabx}
\def\polhk#1{\setbox0=\hbox{#1}{\ooalign{\hidewidth
    \lower1.5ex\hbox{`}\hidewidth\crcr\unhbox0}}}
\newtheorem{Theorem}{Theorem}
\newtheorem{Claim}[Theorem]{Claim}
\newtheorem{Lemma}[Theorem]{Lemma}
\newtheorem{Cor}[Theorem]{Corollary}

\newtheorem{Question}[Theorem]{Question}

\theoremstyle{definition}
\newtheorem{Def}[Theorem]{Definition}

\theoremstyle{remark}

\newcommand{\forces}{\Vdash}
\newcommand{\restrict}{\mathord{\upharpoonright}}

\renewcommand{\c}{\mathfrak{c}}
\renewcommand{\b}{\mathfrak{b}}

\renewcommand{\[}{\left[}
\renewcommand{\]}{\right]}
\renewcommand{\P}{\mathbb{P}}

\newcommand{\Q}{\mathbb{Q}}

\newcommand{\lc}{\left|}
\newcommand{\rc}{\right|}

\newcommand\ZFC{\mathrm{ZFC}}

\newcommand\MA{\mathrm{MA}}
\newcommand\PFA{\mathrm{PFA}}

\newcommand\GCH{\mathrm{GCH}}

\newcommand\CH{\mathrm{CH}}

\newcommand{\BS}{{\omega}^{\omega}}

\DeclareMathOperator{\pred}{pred}
\DeclareMathOperator{\cone}{cone}

\DeclareMathOperator{\suppt}{suppt}
\DeclareMathOperator{\otp}{otp}

\DeclareMathOperator{\add}{add}

\DeclareMathOperator{\dom}{dom}

\DeclareMathOperator{\succc}{succ}

\DeclareMathOperator{\hgt}{ht}
\DeclareMathOperator{\cf}{cf}

\newcommand{\Pset}{\mathcal{P}}

\newcommand{\N}{\mathcal{N}}

\newcommand{\U}{{\mathcal{U}}}

\newcommand{\I}{{\mathcal{I}}}
\newcommand{\II}{{\mathscr{I}}}

\newcommand{\LL}{{\mathbb{L}}}
\newcommand{\F}{{\mathcal{F}}}

\newcommand{\Ta}{{\mathbb{T}}}

\newcommand{\V}{{\mathbf{V}}}
\newcommand{\VG}{{{\mathbf{V}}[G]}}
\newcommand{\VP}{{\mathbf{V}}^{\P}}
\newcommand{\XX}{{\mathcal{X}}}

\newcommand{\Sa}{\mathbb{S}}

\newcommand{\pr}[2]{\langle #1, #2 \rangle}
\newcommand{\seq}[4]{\langle {#1}_{#2}: #2 #3 #4 \rangle}
\begin{document}
\begin{abstract}
We investigate the unbalanced ordinary partition relations of the form $\lambda\rightarrow (\lambda, \alpha)^2$ for various values of the cardinal $\lambda$ and the ordinal $\alpha$.
For example, we show that for every infinite cardinal $\kappa,$ the existence of a ${\kappa}^{+}-$Suslin tree implies ${\kappa}^{+} \not\rightarrow {\left( {\kappa}^{+}, {\log}_{\kappa}({\kappa}^{+}) + 2 \right)}^{2}$.
The consistency of the positive partition relation $\b\rightarrow (\b, \alpha)^2$ for all $\alpha<\omega_1$ for the bounding number $\b$ is also established from large cardinals.
\end{abstract}
\maketitle
\section{Introduction} \label{sec:intro}
We will only consider partitions of pairs into two colors in this paper.
Many positive relations that hold for such partitions tend to fail for more general ones, or hold only in a weakened form, including all of the positive relations considered in this paper. 
The reader who is also interested in more general partitions may consult \cite{partitionsurvey} or \cite{partitionbible} for a survey. 
A wealth of both positive and negative results is now known about partitions of pairs into two colors.
\begin{Def} \label{def:basic}
 For any set $X$ and cardinal $\kappa$, ${\[X\]}^{\kappa} = \{A \subset X: \lc A \rc = \kappa\}$.
 In particular, ${\[X\]}^{2}$ denotes the collection of unordered pairs from $X$.
 
 For any set $\XX$ and a function $c: \XX \rightarrow 2$, ${K}_{0, c}$ denotes $\{A \in \XX: c(A) = 0\}$ and ${K}_{1, c}$ denotes $\{A \in \XX: c(A) = 1\}$.
\end{Def}
We will sometimes omit the subscript $c$ when it is clear from the context.
The following notation, which is known as the \emph{ordinary partition relation}, was first systematically studied in a 1956 paper by Erd{\H o}s and Rado~\cite{ErdRado}.
This concept will be the main focus of our paper.
For any set of ordinals $X$, $\otp(X)$ denotes the order--type of $X$ with its natural well--order.
\begin{Def} \label{def:partsymbol}
 Let $\alpha$, $\beta$, and $\gamma$ be ordinals.
 $\alpha \rightarrow {(\beta, \gamma)}^{2}$ means the following: for any function $c: {\[\alpha\]}^{2} \rightarrow 2$ \emph{either}
 \begin{itemize}
  \item
  there is a set $X \subset \alpha$ such that $\otp(X) = \beta$ and ${\[X\]}^{2} \subset {K}_{0}$,
 \end{itemize}
 \emph{or}
 \begin{itemize}
  \item
  there is a set $X \subset \alpha$ such that $\otp(X) = \gamma$ and ${\[X\]}^{2} \subset {K}_{1}$.
 \end{itemize}
\end{Def}
Using this notation we have the simplest partition relations $6 \rightarrow {\left( 3, 3 \right)}^{2}$ and $5 \not \rightarrow {\left( 3, 3 \right)}^{2}$. 
The well--known Ramsey's theorem can be stated as the relation $\omega \rightarrow {\left( \omega, \omega \right)}^{2}$.
Any na{\"i}ve attempt to generalize Ramsey's theorem to uncountable ordinals fails since, for example,  ${2}^{{\aleph}_{0}} \not\rightarrow {({\omega}_{1}, {\omega}_{1})}^{2}$ via the well-known Sierpinski partition.
Indeed for any cardinal $\kappa > \omega$, $\kappa \rightarrow {(\kappa, \kappa)}^{2}$ if and only if $\kappa$ is weakly compact.
So for smaller cardinals such as ${\omega}_{1}$ one investigates unbalanced  versions $\kappa\rightarrow(\kappa, \alpha)^2$ of this relation.
The first step in this direction was taken by Erd{\H o}s, Dushnik and Miller (see \cite{partitionbible}), who showed that for any cardinal $\kappa \geq \omega$, $\kappa \rightarrow {(\kappa, \omega)}^{2}$.
The following is a basic question guiding research in this area and lies behind most of the work in this paper.
\begin{Question} \label{q:partition1}
  To what extent can the Erd{\H o}s-Dushnik-Miller theorem be improved?
\end{Question}
Erd{\H o}s and Rado (see \cite{partitionbible}) showed that if $\kappa > \omega$ and $\kappa$ is a regular cardinal, then $\kappa \rightarrow {(\kappa, \omega + 1)}^{2}$, providing the first improvement to the Erd{\H o}s- Dushnik-Miller theorem.
The question of whether this theorem of Erd{\H o}s and Rado can be further improved by increasing the ordinal in the second alternative from $\omega + 1$ to $\omega + 2$ has a surprising answer even in the case when $\kappa = {\omega}_{1}$.
Todorcevic~\cite{partition} showed that if $\b = {\omega}_{1}$, then ${\omega}_{1} \not\rightarrow {({\omega}_{1}, \omega + 2)}^{2}$, while $\PFA$ implies that for all $\alpha < {\omega}_{1}$, ${\omega}_{1} \rightarrow {({\omega}_{1}, \alpha)}^{2}$ (see \cite{forcingaxioms}). 
Here $\b$ is the bounding number.
\begin{Def} \label{def:cards1}
  For functions $f, g \in \BS$, $f \; {\leq}^{\ast} \; g$ means $\exists n \in \omega \forall m \geq n \[f(m) \leq g(m)\]$.
  A set $X \subset \BS$ is \emph{unbounded} if there is no $f \in \BS$ such that $\forall g \in X \[g \; {\leq}^{\ast} \; f\]$. 
  $\b$ is the least size of an unbounded subset of $\BS$.
\end{Def}
Under $\PFA$, $\b = {2}^{{\aleph}_{0}} = {\aleph}_{2}$, and Laver~\cite{lav75} showed that $\MA + {2}^{{\aleph}_{0}} = {\aleph}_{2}$ implies ${\omega}_{2} \not\rightarrow {\left( {\omega}_{2}, \omega + 2 \right)}^{2}$ (see also Chapter 13 of \cite{forcingaxioms}).
In light of these results Todorcevic posed the following question in the late 80s.
\begin{Question} \label{q:partition2}
 Is it consistent to have $\b \rightarrow {\left( \b, \alpha \right)}^{2}$, for all $\alpha < {\omega}_{1}$? 
 Can this hold when $\b = {2}^{{\aleph}_{0}}$?
\end{Question}
Another way to pose this question is to ask whether the relation $\b \not\rightarrow {\left(\b, \omega + 2 \right)}^{2}$ is provable without any further assumptions on $\b$.
We give a positive answer to Question \ref{q:partition2} in this paper assuming the consistency of large cardinals.
In Section \ref{sec:measurable} we are able to show the consistency of $\MA(\sigma-\text{linked})$ with the partition relation ${2}^{{\aleph}_{0}} \rightarrow {\left({2}^{{\aleph}_{0}}, \alpha \right)}^{2}$, for all $\alpha < {\omega}_{1}$, starting from a measurable cardinal.

Erd{\H o}s and Rado actually established a stronger from of their above mentioned theorem for uncountable successor cardinals.
They showed (see \cite{partitionbible}) that if $\kappa \geq \omega$ is any cardinal, then ${\kappa}^{+} \rightarrow {\left( {\kappa}^{+}, {\log}_{\kappa}({\kappa}^{+}) + 1 \right)}^{2}$.
The notation ${\log}_{\kappa}({\kappa}^{+})$ is explained in Definition \ref{def:lambda}.
Question \ref{q:partition1} motivates the following problem.
\begin{Question} \label{q:partition3}
 Under what circumstances does the relation ${\kappa}^{+} \rightarrow {\left( {\kappa}^{+}, {\log}_{\kappa}({\kappa}^{+}) + 2 \right)}^{2}$ hold for an infinite cardinal $\kappa$?
 More generally, when does ${\kappa}^{+} \rightarrow {\left( {\kappa}^{+}, \alpha \right)}^{2}$, for all $\alpha < {\kappa}^{+}$?
\end{Question}
If ${2}^{\kappa} = {\kappa}^{+}$, then ${\kappa}^{+} \not\rightarrow {\left( {\kappa}^{+}, \kappa + 2 \right)}^{2}$ (see \cite{partitionsurvey}).
This result was first proved for regular $\kappa$ by Hajnal and then later generalized to all infinite $\kappa$ by Todorcevic.
$\GCH$ is equivalent to the statement that ${\log}_{\kappa}({\kappa}^{+}) = \cf(\kappa)$, for all $\kappa \geq \omega$.
So under $\GCH$, the relation ${\kappa}^{+} \rightarrow {\left( {\kappa}^{+}, {\log}_{\kappa}({\kappa}^{+}) + 2 \right)}^{2}$ fails for every infinite regular $\kappa$.
However this result does not address singular cardinals.
Observe that when $\kappa = {\aleph}_{\omega}$, ${\log}_{\kappa}({\kappa}^{+}) = \omega$, regardless of any assumptions about cardinal arithmetic.
In Section \ref{sec:suslincolor}, we give a further negative result pertaining to Question \ref{q:partition3}.
We show that if $\kappa \geq \omega$ is any cardinal such that there exists a ${\kappa}^{+}$--Suslin tree, then ${\kappa}^{+} \not\rightarrow {\left( {\kappa}^{+}, {\log}_{\kappa}({\kappa}^{+}) + 2 \right)}^{2}$.
In particular, to force the consistency of ${\aleph}_{\omega + 1} \rightarrow {\left( {\aleph}_{\omega + 1}, \omega + 2\right)}^{2}$, one would need to get rid of all ${\aleph}_{\omega + 1}$--Suslin trees.
It is known that this requires some fairly large cardinals.
This theorem is established via a negative partition relation for the comparable pairs of a ${\kappa}^{+}$--Suslin tree, which may be of independent interest.

In this paper, we also consider the following rectangular version of the ordinary partition relation.
\begin{Def} \label{def:inhomo}
  Let $\alpha, \beta, \gamma$, and $\delta$ be ordinals.
  $\alpha \rightarrow {\left(\beta, (\gamma:\delta)\right)}^{2}$ means the following: for any coloring $c: {\[\alpha\]}^{2} \rightarrow 2$ \emph{either}
  \begin{itemize}
   \item
   there is a set $X \subset \alpha$, such that $\otp(X) = \beta$ and ${\[X\]}^{2} \subset {K}_{0}$,
  \end{itemize}
\emph{or}
\begin{itemize}
 \item
 there exist sets $A, B \subset \alpha$ such that:
 \begin{enumerate}
  \item
  $\otp(A) = \gamma$ and $\otp(B) = \delta$;
  \item
  $\forall \zeta \in A \forall \xi \in B \[\zeta < \xi\]$;
  \item
  $\forall \zeta \in A \forall \xi \in B \[\{\zeta, \xi\} \in {K}_{1}\]$.
 \end{enumerate}
\end{itemize}
\end{Def}
The relation $\alpha \rightarrow {\left( \beta, \left(\gamma:\delta\right) \right)}^{2}$ appears weaker than the relation $\alpha \rightarrow {\left( \beta, \gamma + \delta \right)}^{2}$.
While the first alternative is the same for both of them, the second alternative of the rectangular partition relation asks for a homogeneous rectangle of length $\gamma$ and width $\delta$ where the second alternative of the ordinary partition relation asks for a homogeneous square with sides $\gamma + \delta$.
Getting a homogeneous rectangle is generally much easier than getting a homogeneous square.
Laver in fact proved the following stronger statement in his above mentioned paper \cite{lav75}: $\MA + {2}^{{\aleph}_{0}} = {\aleph}_{2}$ implies ${2}^{{\aleph}_{0}} \not\rightarrow {\left( {2}^{{\aleph}_{0}}, (\omega:2) \right)}^{2}$.
Baumgartner~\cite{bau76} proved that if $\GCH$ holds in $\V$ and if $\kappa$ is a regular cardinal in $\V$, then there is a cofinality preserving extension in which ${2}^{{\aleph}_{0}} = {\kappa}^{+}$ and ${2}^{{\aleph}_{0}} \not\rightarrow {\left( {2}^{{\aleph}_{0}}, (\omega:2) \right)}^{2}$.

It is not hard to see that if there is a $\kappa$-complete ideal $\II$ on $\kappa$ such that ${\II}^{+}$ has the c.\@c.\@c.\@, then $\kappa \rightarrow {\left(\kappa, (\omega:2) \right)}^{2}$.
On the other hand, Todorcevic showed in \cite{realpartition} that if $\kappa$ is any cardinal with $\cf(\kappa) > \omega$, then there is a c.\@c.\@c.\@ forcing extension in which $\kappa \not\rightarrow {(\kappa, \omega + 2)}^{2}$.
In particular if $\kappa$ is a measurable cardinal, then in this extension, there is a $\kappa$ complete ideal $\II$ on $\kappa$ with the property that ${\II}^{+}$ is c.\@c.\@c.
Thus he was able to prove that the relation ${2}^{{\aleph}_{0}} \rightarrow {\left( {2}^{{\aleph}_{0}}, (\omega:2) \right)}^{2}$ is consistently strictly weaker than the relation ${2}^{{\aleph}_{0}} \rightarrow {\left( {2}^{{\aleph}_{0}}, \omega + 2 \right)}^{2}$, given the consistency of a measurable cardinal.
Komj{\'a}th~\cite{kom98} produced a model where ${\omega}_{1} \rightarrow {({\omega}_{1}, (\omega:2))}^{2}$ while ${\omega}_{1} \not\rightarrow {({\omega}_{1}, \omega + 2)}^{2}$ without using any large cardinals.
In Section \ref{sec:suslincolor} of this paper, we generalize Komj{\'a}th's result to all successor cardinals.
For each $\kappa \geq \omega$, we produce a model in which ${\kappa}^{+} \rightarrow {({\kappa}^{+}, (\omega:2))}^{2}$ holds and ${\kappa}^{+} \rightarrow {({\kappa}^{+}, \omega + 2)}^{2}$ fails starting only with the consistency of $\ZFC$.
\section{Notation} \label{sec:notation}
We set out some notation that will be used throughout the paper in this section.
``$a \subset b$'' means $\forall x\[x \in a \implies x \in b\]$, so the symbol ``$\subset$'' does not denote proper subset.
$\c$ denotes ${2}^{{\aleph}_{0}}$.
For a set $X$, $\Pset(X)$ is the powerset of $X$.

An ideal $\II$ on a set $X$ is said to be  \emph{$\kappa$-complete} if for any $\delta < \kappa$ and any sequence $\seq{A}{\alpha}{<}{\delta}$ of elements of $\II$, ${\bigcup}_{\alpha < \delta}{{A}_{\alpha}} \in \II$.
${\II}^{+}$ denotes the collection of \emph{$\II$-positive sets} -- that is, ${\II}^{+} = \Pset(X) \setminus \II$.
When we think of ${\II}^{+}$ as a forcing notion it will be understood that the relation is $\subset$.
For a proper non-empty ideal $\II$ on $X$, forcing with $\pr{{\II}^{+}}{\subset}$ is equivalent to forcing with the Boolean algebra $\Pset(X) \slash \II$.
${\II}^{\ast}$ denotes the \emph{dual filter to $\II$} -- that is, ${\II}^{\ast} = \{X \setminus A: A \in \II\}$.

We will frequently make use of elementary submodels.
We will simply write ``$M \prec H{(\theta)}$'' to mean ``$M$ is an elementary submodel of $H{(\theta)}$, where $\theta$ is a regular cardinal that is large enough for the argument at hand''.

We will consider various forcing axioms and also the maximal fragments of these axioms that is consistent with existence of a Suslin tree on ${\omega}_{1}$. 
\begin{Def} \label{def:fa}
Let $\Gamma$ be a class of forcing notions and $\kappa$ a cardinal.
${\MA}_{\kappa}(\Gamma)$ is the following statement: If $\P \in \Gamma$ and $\seq{D}{\alpha}{<}{\kappa}$ is a sequence of dense subsets of $\P$, then there exists a filter $G \subset \P$ such that $\forall \alpha < \kappa \[{D}_{\alpha} \cap G \neq 0\]$.
$\MA(\Gamma)$ is the statement that $\forall \kappa < \c\[{\MA}_{\kappa}(\Gamma) \ \text{holds}\]$.
$\MA$ is $\MA(\Gamma)$ with $\Gamma$ equaling the class of c.\@c.\@c.\@ posets. 
\end{Def}
\begin{Def} \label{def:suslin}
 For a cardinal $\kappa \geq \omega$, $\Sa$ is a \emph{$\kappa$-Suslin tree} if $\Sa$ is a tree of size $\kappa$, but $\Sa$ has no chains or antichains of size $\kappa$.
 \end{Def}
 Given a $\kappa$-Suslin tree $\Sa$, ${\Sa}^{\[2\]}$ denotes $\{\{a, b\} \subset \Sa: a < b \}$.
 \begin{Def} \label{def:mas}
  Let $\Sa$ be an ${\omega}_{1}$-Suslin tree.
  ${\MA}(\Sa)$ is the following statement: if $\P$ is a poset with the property that $\P \times \Sa$ is c.\@c.\@c.\@ and if $\{{D}_{\alpha}: \alpha < \kappa\}$ is a collection of dense subsets of $\P$, where $\kappa$ is a cardinal less than $\c$, then there is a filter $G$ on $\P$ such that $\forall \alpha < \kappa\[G \cap {D}_{\alpha} \neq 0\]$.
 \end{Def}
 \section{A negative partition relation from a Suslin tree} \label{sec:suslincolor}
 In this section, we shed some light on Question \ref{q:partition3} by showing that ${\kappa}^{+} \not\rightarrow {\left( {\kappa}^{+}, {\log}_{\kappa}({\kappa}^{+}) + 2 \right)}^{2}$, if there is a ${\kappa}^{+}$-Suslin tree.
 The logarithmic operation is defined as follows.
 \begin{Def} \label{def:lambda}
 Let $\kappa \geq \omega$ be a cardinal. Define
 \begin{align*}
  {\log}_{\kappa}({\kappa}^{+}) = \min\{\lambda: {\kappa}^{\lambda} > \kappa\}.
 \end{align*}
\end{Def}
Note that ${\log}_{\kappa}({\kappa}^{+})$ is a regular cardinal greater than or equal to $\omega$ and that ${\log}_{\kappa}({\kappa}^{+}) \leq \cf(\kappa)$.
We also consider the following variation:
\begin{Def} \label{def:incfunctions}
 Let $\kappa \geq \omega$ and $\lambda \leq \kappa$ be cardinals.
 \begin{align*}
  {}^{\lambda \uparrow}{\kappa} = \{f: f \ \text{is a strictly increasing function from} \ \lambda \ \ \text{to} \ \kappa\}
 \end{align*}
\end{Def}
The following is easy to check.
We leave its proof to the reader.
\begin{Lemma} \label{lem:incequalsordinary}
Let $\kappa$ and $\lambda$ be cardinals with $\omega \leq \lambda \leq \cf(\kappa) \leq \kappa$.
Then $\lc {}^{\lambda \uparrow}{\kappa}\rc = {\kappa}^{\lambda}$.
In particular, if $\kappa \geq \omega$ is a cardinal and $\lambda = {\log}_{\kappa}({\kappa}^{+})$, then $\lc {}^{\lambda \uparrow}{\kappa} \rc = {\kappa}^{\lambda} > \kappa$.
\end{Lemma}
Our negative partition result is obtained by considering colorings of the pairs of comparable elements of a Suslin tree.
Partition relations involving such colorings were studied by M{\'a}t{\'e}~\cite{matesuslin}, and Todorcevic~\cite{posetpartition} proved similar partition relations for more general partial orders.
The coloring shown to exist in Theorem \ref{thm:coloring} was originally discovered in the special case when ${\kappa}^{+} = {\omega}_{1}$ and exposed in Theorem 28 of \cite{suslincombdic}.  
\begin{Lemma} \label{lem:main2}
 Let $\kappa \geq \omega$ be a cardinal.
 Let $\theta$ be a sufficiently large regular cardinal and fix $x \in H(\theta)$.
 There is $M \prec H(\theta)$ such that
  \begin{enumerate}
   \item
   $\kappa \cup \{x\} \subset M$ and $\lc M \rc = \kappa$;
   \item
   For each $\alpha < {\log}_{\kappa}({\kappa}^{+})$, ${}^{\alpha}{M} \subset M$.
  \end{enumerate}
\end{Lemma}
\begin{proof}
For ease of notation, write $\lambda = {\log}_{\kappa}({\kappa}^{+})$.
 Build a sequence $\langle {M}_{\xi}: \xi < \lambda \rangle$ with the following properties:
 \begin{enumerate}
  \item[(3)]
  $x \in {M}_{0}$ and for each $\xi < \lambda$, ${M}_{\xi} \prec H(\theta)$ with $\kappa \subset {M}_{\xi}$ and $\lc {M}_{\xi} \rc = \kappa$;  
  \item[(4)]
  for all $\zeta < \xi < \lambda$, ${M}_{\zeta} \subset {M}_{\xi}$;
  \item[(5)]
  for each $\alpha < \lambda$, $\forall \lambda > \beta \geq \alpha \[{}^{\alpha}{{M}_{\beta}} \subset {M}_{\beta + 1}\]$.
 \end{enumerate}
Suppose that such a sequence has been built.
Put $M = {\bigcup}_{\xi < \lambda}{{M}_{\xi}}$.
By (3) and (4), $M \prec H(\theta)$ and $\kappa \cup \{x\} \subset M$.
Since $\lambda \leq \kappa$, $\lc M \rc = \kappa$.
Let $\alpha < \lambda$ and fix $f: \alpha \rightarrow M$.
For each $\eta < \alpha$ there is ${\xi}_{\eta} < \lambda$ such that $f(\eta) \in {M}_{{\xi}_{\eta}}$.
By the regularity of $\lambda$, there is $\alpha \leq \xi < \lambda$ such that $\forall \eta < \alpha \[{\xi}_{\eta} \leq \xi\]$.
Therefore, $f \in {}^{\alpha}{{M}_{\xi}} \subset {M}_{\xi + 1} \subset M$.
Thus (1) and (2) hold.

To build such a sequence, proceed as follows.
${M}_{0}$ can be an arbitrary elementary submodel of $H(\theta)$ containing $\kappa \cup \{x\}$ of size $\kappa$.
If $\alpha < \lambda$ and ${M}_{\alpha}$ is given, then note that by the definition of $\lambda$, for all $\xi \leq \alpha$, $\lc {}^{\xi}{{M}_{\alpha}} \rc \leq \kappa$.
Therefore, it is possible to find ${M}_{\alpha + 1} \prec H(\theta)$ such that $\lc {M}_{\alpha + 1} \rc = \kappa$ and ${M}_{\alpha} \cup \left( {\bigcup}_{\xi \leq \alpha}{{}^{\xi}{{M}_{\alpha}}} \right) \subset {M}_{\alpha + 1}$.
When $\alpha < \lambda$ is a limit ordinal, put ${M}_{\alpha} = {\bigcup}_{\xi < \alpha}{{M}_{\xi}}$.
It is clear that this works.
\end{proof}
\begin{Lemma} \label{lem:main1}
Let $\kappa \geq \omega$ be a cardinal and suppose that $\Sa$ is a ${\kappa}^{+}$-Suslin tree.
Let $\theta$ be a sufficiently large regular cardinal.
Fix $M \prec H(\theta)$ such that $\kappa \subset M$, $\lc M \rc = \kappa$, and $M$ contains all the relevant objects.
Put $\delta = M \cap {\kappa}^{+}$.
Let $L \subset \Sa$ be such that $\{\hgt(s): s \in M \cap L\}$ is unbounded in $\delta$.
Let $D \subset \Sa$ with $D \in M$.
Assume that $\forall s \in L \cap M \exists t \in D \[t \geq s\]$.
Then there is $s \in L \cap M$ such that $D$ is dense above $s$ in $\Sa$.
Moreover, if there is $t \in \Sa$ such that $L = \pred(t)$, then $\{\hgt(s): s \in L \cap M \cap D\}$ is unbounded in $\delta$.
\end{Lemma}
\begin{proof}
 Put $E = \{x \in \Sa: {\cone}_{D}(x) = 0\}$.
  $E \in M$.
  So there exists $A \in M$ such that $A \subset E$, $A$ is an antichain, and $A$ is maximal with respect to these two properties.
  As $A$ has size at most $\kappa$, find $\alpha < \delta$ such that $A \subset {\Sa}_{< \alpha}$.
  Let $x \in L \cap M$ be such that $\hgt(x) \geq \alpha$.
  If $D$ is not dense above $x$ in $\Sa$, then there is $s \in \Sa$ such that $s \geq x$ and ${\cone}_{D}(s) = 0$.
  Thus $s \in E$ and is comparable to some $a \in A$.
  It follows that $a \leq x$.
  However, by hypothesis, there is $y \in D$ with $x \leq y$.
  $y \in {\cone}_{D}(a)$, contradicting $a \in E$.

  For the second statement assume that $L = {\pred}_{\Sa}(s)$ for some $s \in \Sa$, and fix $\alpha < \delta$.
  By the first statement, fix $x \in L \cap M$ such that $D$ is dense above $x$ in $\Sa$.
  Note that ${\cone}_{D}(x) \in M$ and that it is a set of size ${\kappa}^{+}$.
  Put $B = \{y \in {\cone}_{D}(x): \hgt(y) > \alpha \} \in M$.
  Choose $A \in M$ such that $A \subset B$, $A$ is an antichain, and $A$ is maximal with respect to these two properties.
  As $A$ has size at most $\kappa$, fix $\beta < \delta$ such that $A \subset {\Sa}_{< \beta}$.
  Fix $t \in L \cap M$ with $\hgt(t) > \max\{\alpha, \beta, \hgt(x)\}$.
  Thus $t \geq x$ and there is $y \in D$ with $y \geq t$.
  Since $y \in B$, there is $a \in A$ such that $a \leq y$.
  It follows that $a \leq t \leq s$.
  Therefore, $a \in L \cap D \cap M$ and $\hgt(a) > \alpha$.
\end{proof}
\begin{Theorem} \label{thm:coloring}
 Let $\kappa \geq \omega$ be a cardinal.
 Assume that $\Sa$ is a ${\kappa}^{+}$-Suslin tree.
 Furthermore, assume that for each $s \in \Sa$, $\succc(s) = \{t \in \Sa: t \geq s \ \text{and} \ \hgt(t) = \hgt(s) + 1\}$ has size exactly equal to $\kappa$ and also that $\lc \{t \in \Sa: t \geq s \}\rc = {\kappa}^{+}$.
 Then there is a coloring $c: {\Sa}^{\[2\]} \rightarrow 2$ such that
 \begin{enumerate}
  \item
  There is no $X \subset \Sa$ such that $\lc X \rc = {\kappa}^{+}$ and ${X}^{\[2\]} \subset {K}_{0}$;
  \item
  There is no $s \in \Sa$ and $B \subset \pred(s)$ such that $\otp(B) = {\log}_{\kappa}({\kappa}^{+}) + 2$ and ${B}^{\[2\]} \subset {K}_{1}$.
 \end{enumerate}
\end{Theorem}
\begin{proof}
Write $\lambda$ for ${\log}_{\kappa}({\kappa}^{+})$
 For $f, g \in {}^{\lambda}{\kappa}$, if $f \neq g$, let $\Delta(f, g)$ denote the least $\alpha < \lambda$ such that $f(\alpha) \neq g(\alpha)$.
 Choose a collection $\{{f}_{s}: s \in \Sa\}$ of ${\kappa}^{+}$-many pairwise distinct elements of ${}^{\lambda \uparrow}{\kappa}$.
 This is possible since by Lemma \ref{lem:incequalsordinary}, $\lc {}^{\lambda \uparrow}{\kappa}\rc = {\kappa}^{\lambda} > \kappa$.
 For each $s \in \Sa$, let $\{{s}^{+}_{\beta}: \beta \in \kappa\}$ be a 1-1 enumeration of $\succc(s)$.
 Now, define $c: {\Sa}^{\[2\]} \rightarrow 2$ as follows.
 For any pair $s, t \in \Sa$, if $s < t$, then there is a unique $\beta \in \kappa$ such that ${s}^{+}_{\beta} \leq t$.
 If ${f}_{t}(\Delta({f}_{s}, {f}_{t})) = \beta$, then set $c(\{s, t\}) = 1$.
 Otherwise $c(\{s, t\}) = 0$.
 The first claim will establish (1).
 \begin{Claim} \label{claim:coloring1}
  There is no $X \in {\[\Sa\]}^{{\kappa}^{+}}$ such that ${X}^{\[2\]} \subset {K}_{0}$.
 \end{Claim}
 \begin{proof}
 Suppose not.
 Fix a counterexample $X$.
 Let $\theta$ be a sufficiently large regular cardinal.
 Let $M \prec H(\theta)$ be of size $\kappa$ such that $\forall \alpha < \lambda \[{}^{\alpha}{M} \subset M\]$ and $\kappa \cup \{ \langle \Sa, {<}_{\Sa} \rangle, \langle {f}_{s}: s \in \Sa \rangle, \langle \langle {s}^{+}_{\beta}: \beta \in \kappa \rangle: s \in \Sa \rangle, c, X \} \subset M$.
 $X \setminus M$ is a subset of $\Sa$ of size ${\kappa}^{+}$.
 So it is not an antichain.
 Fix $t, u \in X \setminus M$ such that $t < u$.
 We will get a contradiction if we can show that ${f}_{t} = {f}_{u}$.
 To this end, fix $\alpha < \lambda$ and assume that ${f}_{t}(\xi) = {f}_{u}(\xi)$, for all $\xi < \alpha$.
 Let $\sigma = {f}_{t} \restrict \alpha = {f}_{u} \restrict \alpha$.
 By the closure properties of $M$, $\sigma \in M$.
 Consider any $s \in {\pred}_{X}(t) \cap M$.
 Let $\beta(s)$ denote the unique $\beta \in \kappa$ such that ${s}^{+}_{\beta} \leq t$.
 If $\sigma = {f}_{s} \restrict \alpha$ and if $\beta(s) = {f}_{t}(\alpha)$, then since $c(\{s, t\}) = 0$, it follows that ${f}_{s}(\alpha) = {f}_{t}(\alpha)$.
 Put $\zeta = {f}_{t}(\alpha)$ and ${\zeta}^{\ast} = {f}_{u}(\alpha)$.
 Let
 \begin{align*}
  D = \left\{v \in \Sa: \exists s \in X \[{f}_{s} \restrict \alpha = \sigma \ \text{and} \ {f}_{s}(\alpha) = {\zeta}^{\ast} \ \text{and} \ {s}^{+}_{\zeta} = v\] \right\}.
 \end{align*}
 It is easy to see that $D \in M$.
 Let $L = \pred(t)$.
 Note that ${u}^{+}_{\zeta} \in D$.
 Therefore, $\forall y \in L \cap M \exists v \in D\[y \leq v\]$.
 So by Lemma \ref{lem:main1}, we can find $v \in L \cap M \cap D$.
 Let $s \in X$ be such that ${f}_{s} \restrict \alpha = \sigma$, ${f}_{s}(\alpha) = {\zeta}^{\ast}$ and ${s}^{+}_{\zeta} = v$.
 Then $s < t$ and $s \in M \cap {\pred}_{X}(t)$.
 Since ${s}^{+}_{\zeta} = v \leq t$, $\beta(s) = \zeta = {f}_{t}(\alpha)$.
 Thus ${f}_{u}(\alpha) = {\zeta}^{\ast} = {f}_{s}(\alpha) = {f}_{t}(\alpha)$.
 So by induction on $\alpha \in \lambda$, $\forall \alpha \in \lambda \[{f}_{u}(\alpha) = {f}_{t}(\alpha)\]$, which is a contradiction.
 \end{proof}
 We next work toward showing that (2) holds.
 We need a few preliminary claims.
 Aiming for a contradiction, fix $s \in \Sa$ and $B \subset \pred(s)$ such that $\otp(B) = \lambda + 2$ and ${B}^{\[2\]} \subset {K}_{1}$.
 For each $\alpha < \lambda + 2$, let $s(\alpha)$ denote the $\alpha$-th element of $B$, and for each $\alpha < \lambda$, let $\beta(\alpha)$ be the unique $\beta < \kappa$ such that ${\left( s(\alpha) \right)}^{+}_{\beta} \in \pred(s)$.
 Using the partition relation $\lambda \rightarrow {\left( \lambda, \omega \right)}^{2}$ and the regularity of $\lambda$, we conclude that there is $A \in {\[\lambda\]}^{\lambda}$ such that:
 \begin{align*}
& \text{\emph{either} there exists} \ \beta < \kappa \ \text{such that for all} \ \alpha \in A, \beta(\alpha) = \beta \\
& \text{\emph{or} for all} \ \xi, \alpha \in A \ \text{if} \ \xi < \alpha, \ \text{then} \ \beta(\xi) < \beta(\alpha).
 \end{align*}
\begin{Claim} \label{claim:coloring2}
 For each $\sigma \in {\kappa}^{< \lambda}$, there is no set $X \in {\[\kappa\]}^{\lambda}$ such that $\forall \gamma \in X \exists {\alpha}_{\gamma} \in A\[{f}_{s\left({\alpha}_{\gamma}\right)} \supset {\sigma}^{\frown}{\langle \gamma \rangle}\]$.  
 \end{Claim}
 \begin{proof}
 Suppose not.
 Note that for all ${\gamma}_{0}, {\gamma}_{1} \in X$, if ${\alpha}_{{\gamma}_{0}} = {\alpha}_{{\gamma}_{1}}$, then ${\gamma}_{0} = {\gamma}_{1}$.
 Therefore, applying again the partition relation $\lambda \rightarrow {\left( \lambda, \omega \right)}^{2}$ and the regularity of $\lambda$, there is $Y \in {\[X\]}^{\lambda}$ such that $\forall {\gamma}_{0}, {\gamma}_{1} \in Y \[{\gamma}_{0} < {\gamma}_{1} \implies {\alpha}_{{\gamma}_{0}} < {\alpha}_{{\gamma}_{1}}\]$.
 Fix ${\gamma}_{0} < {\gamma}_{1} < {\gamma}_{2}$ in $Y$.
 Then $s\left({\alpha}_{{\gamma}_{0}}\right) < s\left({\alpha}_{{\gamma}_{1}}\right) < s\left({\alpha}_{{\gamma}_{2}}\right)$ are members of $B$.
 As $c\left( \left\{ s\left({\alpha}_{{\gamma}_{0}}\right), s\left({\alpha}_{{\gamma}_{1}}\right) \right\} \right) = 1$ and as ${f}_{s\left( {\alpha}_{{\gamma}_{0}} \right)} \supset {\sigma}^{\frown}{\langle {\gamma}_{0} \rangle}$ and ${f}_{s\left( {\alpha}_{{\gamma}_{1}} \right)} \supset {\sigma}^{\frown}{\langle {\gamma}_{1} \rangle}$, $\beta({\alpha}_{{\gamma}_{0}}) = {\gamma}_{1}$.
 Similarly, $\beta({\alpha}_{{\gamma}_{0}}) = {\gamma}_{2}$.
 This a contradiction.
\end{proof}
For each $\sigma \in {\kappa}^{< \lambda}$, let ${A}_{\sigma} = \{\alpha \in A: \sigma \subset {f}_{s(\alpha)}\}$.
\begin{Claim} \label{claim:coloring3}
 Let $\sigma, \tau \in {\kappa}^{< \lambda}$ such that both ${A}_{\sigma}$ and ${A}_{\tau}$ have size $\lambda$.
 Then $\sigma$ and $\tau$ are comparable.
\end{Claim}
\begin{proof}
Suppose not.
Let $\delta$ be the least ordinal less than $\min\{\dom(\sigma), \dom(\tau)\}$ such that $\sigma(\delta) \neq \tau(\delta)$.
Consider any $\alpha \in {A}_{\sigma}$.
There is ${\alpha}' \in {A}_{\tau}$ with $\alpha < {\alpha}'$.
Since $c\left( \left\{s(\alpha), s({\alpha}') \right\} \right) = 1$ and since $\sigma \subset {f}_{s(\alpha)}$ and $\tau \subset {f}_{s({\alpha}')}$, it follows that $\tau(\delta) = \beta(\alpha)$.
Thus we have $\forall \alpha \in {A}_{\sigma} \[\beta(\alpha) = \tau(\delta)\]$.
By a similar reasoning, we get $\forall \alpha \in {A}_{\tau}\[\beta(\alpha) = \sigma(\delta)\]$.
However, this contradicts the way $A$ was chosen.
\end{proof}
\begin{Claim} \label{claim:coloring4}
 For each $\delta < \lambda$, there is a unique $\sigma \in {\kappa}^{\delta}$ such that $\lc A \setminus {A}_{\sigma} \rc < \lambda$.
\end{Claim}
\begin{proof}
 The uniqueness of $\sigma$ follows from Claim \ref{claim:coloring3}.
 To prove existence, we proceed by induction on $\delta$.
 If $\delta = 0$, then there is nothing to prove.
 Suppose that $\delta < \lambda$ and $\sigma \in {\kappa}^{\delta}$ has the property that $\lc A \setminus {A}_{\sigma} \rc < \lambda$.
 Clearly, $\lc {A}_{\sigma} \rc = \lambda$ and ${A}_{\sigma} = {\bigcup}_{\gamma \in \kappa}{{A}_{{\sigma}^{\frown}{\langle \gamma \rangle}}}$.
 By Claim \ref{claim:coloring2}, $X = \left\{\gamma \in \kappa: {A}_{{\sigma}^{\frown}{\langle \gamma \rangle}} \neq 0 \right\}$ has size less than $\lambda$.
 So by the regularity of $\lambda$, there must be $\gamma \in X$ such that $\lc {A}_{{\sigma}^{\frown}{\langle \gamma \rangle}} \rc = \lambda$.
 By Claim \ref{claim:coloring3}, for any ${\gamma}^{\ast} \in X \setminus \{\gamma\}$, $\lc {A}_{{\sigma}^{\frown}{\langle {\gamma}^{ \ast}\rangle}} \rc < \lambda$.
 Therefore, $A \setminus {A}_{{\sigma}^{\frown}{\langle \gamma \rangle}} = \left( A \setminus {A}_{\sigma} \right) \cup \left({\bigcup}_{{\gamma}^{\ast} \in X \setminus \{\gamma\}} { {A}_{{\sigma}^{\frown}{\langle {\gamma}^{\ast} \rangle}}}\right)$ is a set of size less than $\lambda$.
 Finally, suppose that $\delta < \lambda$ is a limit ordinal and that for each $\mu < \delta$ there is a ${\sigma}_{\mu} \in {\kappa}^{\mu}$ such that $\lc A \setminus {A}_{{\sigma}_{\mu}} \rc < \lambda$.
 By Claim \ref{claim:coloring3}, the ${\sigma}_{\mu}$ form a chain and so $\sigma = {\bigcup}_{\mu < \delta}{{\sigma}_{\mu}} \in {\kappa}^{\delta}$.
As $A \setminus {A}_{\sigma} = {\bigcup}_{\mu < \delta}{\left( A \setminus {A}_{{\sigma}_{\mu}} \right)}$, it follows from the regularity of $\lambda$ that $\lc A \setminus {A}_{\sigma} \rc < \lambda$.
\end{proof}
For each $\delta < \lambda$, let ${\sigma}_{\delta}$ be the unique member of ${\kappa}^{\delta}$ such that $\lc A \setminus {A}_{{\sigma}_{\delta}} \rc < \lambda$.
Again the ${\sigma}_{\delta}$ form a chain and so $f = {\bigcup}_{\delta < \lambda}{{\sigma}_{\delta}} \in {\kappa}^{\lambda}$.
Moreover, it is easy to see that $f$ is strictly increasing.
There is at most one $\alpha \in A$ such that $f = {f}_{s(\alpha)}$.
Therefore, using again the partition relation $\lambda \rightarrow {\left( \lambda, \omega \right)}^{2}$ and the regularity of $\lambda$, it is possible to find ${A}^{\ast} \in {\[A\]}^{\lambda}$ such that: 
\begin{enumerate}
 \item[(3)]
 $\forall \alpha \in {A}^{\ast}\[f \neq {f}_{s(\alpha)}\]$;
 \item[(4)]
 \emph{either} there is a fixed $\delta < \lambda$ such that $\forall \alpha \in {A}^{\ast}\[\Delta\left(f, {f}_{s(\alpha)}\right) = \delta\]$ \emph{or} for all $\alpha, {\alpha}^{\ast} \in {A}^{\ast}$ with $\alpha < {\alpha}^{\ast}$, $\Delta\left(f, {f}_{s(\alpha)}\right) < \Delta\left(f, {f}_{s({\alpha}^{\ast})}\right)$. 
\end{enumerate}
If there is a $\delta < \lambda$ such that $\forall \alpha \in {A}^{\ast}\[\Delta \left( f, {f}_{s(\alpha)}\right) = \delta\]$, then ${A}^{\ast} \subset A \setminus {A}_{{\sigma}_{\delta + 1}}$, contradicting the fact that $\lc A \setminus {A}_{{\sigma}_{\delta + 1}}\rc < \lambda$.
Therefore, for all $\alpha, {\alpha}^{\ast} \in {A}^{\ast}$, if $\alpha < {\alpha}^{\ast}$, then $\Delta\left( f, {f}_{s(\alpha)} \right) < \Delta \left(f, {f}_{s({\alpha}^{\ast})}\right)$.

Now suppose first that there is a fixed $\beta < \kappa$ such that $\forall \alpha \in {A}^{\ast}\[\beta(\alpha) = \beta\]$.
Choose ${\alpha}_{0} < {\alpha}_{1} < {\alpha}_{2}$, all in ${A}^{\ast}$.
It is clear that $\Delta\left({f}_{s({\alpha}_{0})}, {f}_{s({\alpha}_{2})}\right) = \Delta\left(f, {f}_{s({\alpha}_{0})}\right)$ and that $\Delta\left({f}_{s({\alpha}_{1})}, {f}_{s({\alpha}_{2})}\right) = \Delta\left(f, {f}_{s({\alpha}_{1})}\right)$.
Therefore, 
\begin{align*}
& \beta = {f}_{s({\alpha}_{2})}\left(\Delta\left({f}_{s({\alpha}_{0})}, {f}_{s({\alpha}_{2})}\right)\right) = f\left(\Delta\left(f, {f}_{s({\alpha}_{0})}\right)\right) \ \text{and also} \\
& \beta = {f}_{s({\alpha}_{2})}\left(\Delta\left({f}_{s({\alpha}_{1})}, {f}_{s({\alpha}_{2})}\right)\right) = f\left(\Delta\left(f, {f}_{s({\alpha}_{1})}\right)\right).
\end{align*}
This is a contradiction because $f$ is strictly increasing.

In the other case, that is, when $\forall \alpha, {\alpha}^{\ast} \in {A}^{\ast}\[\alpha < {\alpha}^{\ast} \implies \beta(\alpha) < \beta({\alpha}^{\ast})\]$, we argue as follows.
There must be $\delta \in \{\lambda, \lambda + 1\}$ such that $f \neq {f}_{s(\delta)}$.
Choose $\alpha \in {A}^{\ast}$ with $\Delta\left(f, {f}_{s(\alpha)}\right) > \Delta\left(f, {f}_{s(\delta)}\right)$ and $\beta(\alpha) \neq {f}_{s(\delta)}\left(\Delta\left(f, {f}_{s(\delta)}\right)\right)$.
Since $c(\{s(\alpha), s(\delta)\}) = 1$, we have $\beta(\alpha) = {f}_{s(\delta)}\left(\Delta\left({f}_{s(\alpha)}, {f}_{s(\delta)}\right)\right) = {f}_{s(\delta)}\left(\Delta\left(f, {f}_{s(\delta)}\right)\right) \neq \beta(\alpha)$.
This is a contradiction which completes the proof.
\end{proof}
\begin{Theorem} \label{thm:maincor}
 Let $\kappa \geq \omega$ be a cardinal.
 If there is a ${\kappa}^{+}$-Suslin tree, then 
 \begin{align*}
  {\kappa}^{+} \not\rightarrow {\left( {\kappa}^{+}, {\log}_{\kappa}({\kappa}^{+}) + 2 \right)}^{2}.
\end{align*}  
\end{Theorem}
\begin{proof}
 If there is a ${\kappa}^{+}$-Suslin tree, then there is a ${\kappa}^{+}$-Suslin tree $\Sa$ that satisfies the hypotheses of Theorem \ref{thm:coloring}.
 Fix a coloring $c: {\Sa}^{\[2\]} \rightarrow 2$ such that (1) and (2) of Theorem \ref{thm:coloring} are satisfied.
 As $\lc \Sa \rc = {\kappa}^{+}$, let $\langle {s}_{\alpha}: \alpha < {\kappa}^{+}\rangle$ be a 1-1 enumeration of $\Sa$.
 Define $d: {\[{\kappa}^{+}\]}^{2} \rightarrow 2$ as follows.
 For $\alpha, \beta \in {\kappa}^{+}$ with $\alpha < \beta$, if ${s}_{\alpha} < {s}_{\beta}$ and if $c\left( \left\{ {s}_{\alpha}, {s}_{\beta} \right\} \right) = 1$, then define $d(\{\alpha, \beta\}) = 1$.
 Otherwise, define $d(\{\alpha, \beta\}) = 0$.
 
 To see that the first alternative fails, fix $A \subset {\kappa}^{+}$ of size ${\kappa}^{+}$.
 Define $k: {\[A\]}^{2} \rightarrow 2$ as follows.
 Given $\alpha, \beta \in A$ with $\alpha < \beta$, if ${s}_{\beta} < {s}_{\alpha}$, then put $k(\{\alpha, \beta\}) = 1$.
 Otherwise, put $k(\{\alpha, \beta\}) = 0$.
 Applying the partition relation ${\kappa}^{+} \rightarrow {\left( {\kappa}^{+}, \omega \right)}^{2}$, if there is a set $C \subset A$ with $\otp(C) = \omega$ and $k''{\[C\]}^{2} = \{1\}$, then we get a infinite strictly decreasing sequence of elements of $\Sa$, which is not possible.
 So there must be $C \in {\[A\]}^{{\kappa}^{+}}$ such that $k''{\[C\]}^{2} = \{0\}$.
 Now, $X = \{{s}_{\alpha}: \alpha \in C\}$ is a subset of $\Sa$ of size ${\kappa}^{+}$, and so ${X}^{\[2\]} \not\subset {K}_{0, c}$.
 Suppose $s, t \in X$ with $s < t$ and $\{s, t\} \notin {K}_{0, c}$.
 Let $\alpha, \beta \in C$ be such that $s = {s}_{\alpha}$ and $t = {s}_{\beta}$.
 Obviously, $\alpha \neq \beta$.
 If $\beta < \alpha$, then $k(\{\beta, \alpha\}) = 1$, contradicting the choice of $C$.
 Therefore, $\alpha < \beta$, ${s}_{\alpha} < {s}_{\beta}$, and $c\left( \left\{ {s}_{\alpha}, {s}_{\beta} \right\}\right) = 1$.
 This implies that $d(\{\alpha, \beta\}) = 1$.
 Thus ${\[A\]}^{2} \not\subset {K}_{0, d}$.
 
 To see that the second alternative also fails, fix $A \subset {\kappa}^{+}$ whose order-type is ${\log}_{\kappa}({\kappa}^{+}) + 2$, and suppose that $d''{\[A\]}^{2} = \{1\}$.
 In particular, this implies that for any $\alpha, \beta \in A$, if $\alpha < \beta$, then ${s}_{\alpha} < {s}_{\beta}$.
 Therefore, $B = \{{s}_{\alpha}: \alpha \in A\}$ is a subset of $\Sa$ of order-type ${\log}_{\kappa}({\kappa}^{+}) + 2$ and $B \subset \pred\left({s}_{\max(A)}\right)$.
 So ${B}^{\[2\]} \not\subset {K}_{1, c}$.
 Suppose $s, t \in B$, $s < t$, and that $\{s, t\} \notin {K}_{1, c}$.
 Let $\alpha, \beta \in A$ be such that $s = {s}_{\alpha}$ and $t = {s}_{\beta}$.
 Clearly, $\alpha < \beta$, and since $d(\{\alpha, \beta \}) = 1$, it follows that $c(\{s, t\}) = 1$.
 This is a contradiction which completes the proof.
 \end{proof}
 Here are some immediate corollaries of our theorem.
 First, if $\cf(\kappa) = \omega$, then ${\log}_{\kappa}({\kappa}^{+}) = \omega$.
 Therefore we get the following at ${\aleph}_{\omega + 1}$.
 \begin{Cor} \label{cor:omega+1}
  If there is an ${\aleph}_{\omega + 1}$-Suslin tree, then ${\aleph}_{\omega + 1} \not\rightarrow {\left( {\aleph}_{\omega + 1}, \omega + 2 \right)}^{2}$.
 \end{Cor}
 It is known that powerful large cardinal hypotheses are needed to destroy all Suslin trees at ${\aleph}_{\omega + 1}$.
 The current best result due to Sinapova~\cite{sin12} uses $\omega$ many supercompact cardinals.
 
 For ${\omega}_{2}$ we get the following picture: if there is a ${\omega}_{2}$-Suslin tree and ${2}^{\omega} \geq {\omega}_{2}$, then ${\omega}_{2} \not\rightarrow {({\omega}_{2}, \omega + 2)}^{2}$; if there is a ${\omega}_{2}$-Suslin tree and ${2}^{\omega} = {\omega}_{1}$, then ${\omega}_{2} \not\rightarrow {({\omega}_{2}, {\omega}_{1} + 2)}^{2}$.
 The consistency of the relation ${\omega}_{2} \rightarrow {\left( {\omega}_{2}, {\omega}_{1} + 2 \right)}^{2}$ is unknown at present.
 \begin{Question} \label{q:partition4}
  Is it consistent that ${\omega}_{2} \rightarrow {\left( {\omega}_{2}, \alpha \right)}^{2}$, for all $\alpha < {\omega}_{2}$?
  If so, then what is its consistency strength?
 \end{Question}
 By Sierpinski's classical coloring, ${\omega}_{2} \rightarrow {({\omega}_{1}, {\omega}_{1})}^{2}$ implies $\CH$.
 So by our Theorem \ref{thm:maincor}, a positive answer to Question \ref{q:partition4} requires both $\CH$ and the non-existence of ${\omega}_{2}$-Suslin trees.
 A model where these requirements are satisfied was constructed by Laver and Shelah~\cite{aleph2suslin} using a weakly compact cardinal. It is  unknown whether the weakly compact cardinal is needed for their result but it is for the stronger version of their result claiming the consistency of CH with the statement that all $\aleph_2$-Aronszajn trees are special (see, for example, Corollary 7.2.15 of \cite{walks})
 It can be shown, however, that ${\omega}_{2}\not \rightarrow {\left( {\omega}_{2}, {\omega}_{1} + 2 \right)}^{2}$ holds in their model. This does not rule out that their scheme of collapsing a large cardinal to $\aleph_2 $ and then iteratively destroying all counterexamples to ${\omega}_{2}\rightarrow {\left( {\omega}_{2}, {\omega}_{1} + 2 \right)}^{2}$ can not be applied. 
 
 We next use Theorem \ref{thm:maincor} to generalize a result of Komj{\'a}th~\cite{kom98} to all successor cardinals.
 Our result distinguishes between the relations ${\kappa}^{+} \rightarrow {\left( {\kappa}^{+}, \omega + 2 \right)}^{2}$ and ${\kappa}^{+} \rightarrow {\left( {\kappa}^{+}, (\omega:2) \right)}^{2}$ without the use of any large cardinals.
 \begin{Lemma} \label{lem:ultrafilter}
  Let $\kappa \geq {\omega}_{1}$ be a cardinal.
  Suppose $c: {\[\kappa\]}^{2} \rightarrow 2$ is a coloring.
  Suppose there exists a sequence $\langle {F}_{\alpha}: \alpha < \omega + \omega \rangle$ such that:
  \begin{enumerate}
   \item
   $\exists n \in \omega \setminus \{0\} \forall \alpha < \omega + \omega\[ {F}_{\alpha} \in {\[\kappa\]}^{n}\]$;
   \item
   for each $\alpha < \beta < \omega + \omega$, $\max({F}_{\alpha}) < \min({F}_{\beta})$, and there exist $\zeta \in {F}_{\alpha}$ and $\xi \in {F}_{\beta}$ such that $c(\{\zeta, \xi\}) = 1$.
  \end{enumerate}
  Then there exist $A \subset \kappa$ and $\gamma, \delta \in \kappa$ such that:
  \begin{enumerate}
   \item[(3)]
   $\otp(A) = \omega$ and $\forall \xi \in A \[\xi < \gamma < \delta\]$;
   \item[(4)]
   $\forall \xi \in A \[c(\{\xi, \gamma\}) = c(\{\xi, \delta\}) = 1\]$.
  \end{enumerate}
 \end{Lemma}
 \begin{proof}
  Let $\U$ be an ultrafilter on $\omega$.
  For each $i < n$ and $\alpha < \omega + \omega$, let ${F}_{\alpha}(i)$ denote the $i$th element of ${F}_{\alpha}$.
  By (2), for each $\omega \leq \gamma < \omega + \omega$, there are $\langle {i}_{\gamma}, {j}_{\gamma}\rangle \in n \times n$ and ${A}_{\gamma} \in \U$ such that $\forall k \in {A}_{\gamma}\[c\left( \left\{ {F}_{k}({i}_{\gamma}), {F}_{\gamma}({j}_{\gamma})\right\}\right) = 1\]$. 
  Find $\omega \leq \gamma < \delta < \omega + \omega$ such that $\langle {i}_{\gamma}, {j}_{\gamma} \rangle =  \langle {i}_{\delta}, {j}_{\delta} \rangle = \langle i, j \rangle$, for some $\langle i, j \rangle \in n \times  n$.
  Let $A = \{{F}_{k}(i): k \in {A}_{\gamma} \cap {A}_{\delta}\}$.
  Then it is clear that $A$, ${F}_{\gamma}(j)$, and ${F}_{\delta}(j)$ are as needed.
 \end{proof}
 \begin{Theorem} \label{thm:difference}
  Let $\kappa \geq \omega$ be a cardinal and let $\Sa$ be an ${\omega}_{1}$-Suslin tree.
  Assume that ${2}^{\omega} = {\kappa}^{++}$, that $\MA(\Sa)$ holds, and that there exists a ${\kappa}^{+}$-Suslin tree. 
  Then ${\kappa}^{+} \rightarrow {\left({\kappa}^{+}, (\omega:2) \right)}^{2}$, while ${\kappa}^{+} \not\rightarrow {({\kappa}^{+}, \omega + 2)}^{2}$.
 \end{Theorem}
 \begin{proof}
 The existence of a ${\kappa}^{+}$-Suslin tree and the hypothesis that ${2}^{\omega} \geq {\kappa}^{+}$ ensure ${\kappa}^{+} \not\rightarrow {({\kappa}^{+}, \omega + 2)}^{2}$.
  Suppose for a contradiction that $c: {\[{\kappa}^{+}\]}^{2} \rightarrow 2$ is a coloring that witnesses ${\kappa}^{+} \not\rightarrow {\left({\kappa}^{+}, (\omega:2) \right)}^{2}$.
  We find a c.\@c.\@c.\@ poset ${\P}_{c}$ preserving $\Sa$ such that any filter meeting a certain collection of ${\kappa}^{+}$ many dense subsets of ${\P}_{c}$ yields a 0-homogeneous subset of ${\kappa}^{+}$ of size ${\kappa}^{+}$.
  This will give a contradiction.
  \begin{Claim} \label{claim:diff1}
  There is a $\delta < {\kappa}^{+}$ such that for any $F \in {\[{\kappa}^{+}\]}^{< \omega}$, if $F \cap \delta = 0$, then $\forall \alpha < {\kappa}^{+} \exists {\kappa}^{+} > \beta > \alpha \forall \xi \in F \[\xi < \beta \wedge c(\{\xi, \beta\}) = 0\]$.
  \end{Claim}
  \begin{proof}
  If not, then it is possible to build a sequence $\langle {F}_{\alpha}: \alpha < {\kappa}^{+}\rangle$ such that for each $\alpha < {\kappa}^{+}$
  \begin{enumerate}
   \item
   ${F}_{\alpha} \in {\[{\kappa}^{+}\]}^{< \omega}$ and ${F}_{\alpha} \neq 0$;
   \item
   $\forall \mu < \alpha \[\max({F}_{\mu}) < \min({F}_{\alpha})\]$;
   \item
   $\forall \mu < \alpha \exists \zeta \in {F}_{\mu} \exists \xi \in {F}_{\alpha}\[c(\{\zeta, \xi\}) = 1\]$.
  \end{enumerate}
 Now, there are $n \in \omega \setminus \{0\}$ and $X \subset {\kappa}^{+}$ such that $\otp(X) = \omega + \omega$ and $\forall \alpha \in X\[\lc {F}_{\alpha} \rc = n\]$.
 However, by Lemma \ref{lem:ultrafilter}, this contradicts the supposition that $c$ is witness to ${\kappa}^{+} \not\rightarrow {\left({\kappa}^{+}, (\omega:2) \right)}^{2}$.
 \end{proof}
 Fix $\delta$ as in the claim above.
 Put ${\P}_{c} = \{F \in {\[{\kappa}^{+}\]}^{< \omega}: F \cap \delta = 0 \ \wedge {\[F\]}^{2} \subset {K}_{0}\}$.
 For $F, G \in {\P}_{c}$, $G \leq F$ if $G \supset F$.
 \begin{Claim} \label{claim:ccc}
 ${\P}_{c}$ is c.\@c.\@c.\@ and preserves $\Sa$. 
 \end{Claim}
 \begin{proof}
 It is sufficient to prove that $\Sa \times {\P}_{c}$ is c.\@c.\@c.
 Suppose that $\langle \langle {s}_{\alpha}, {F}_{\alpha} \rangle: \alpha < {\omega}_{1} \rangle$ is a sequence of pairwise incompatible elements of $\Sa \times {\P}_{c}$.
 As $\Sa$ is c.\@c.\@c.\@, there is a $(\V, \Sa)$-generic filter $G$ such that $\{\alpha < {\omega}_{1}: {s}_{\alpha} \in G\}$ is uncountable.
 Work in $\V\[G\]$.
 Let $X = \{\alpha < {\omega}_{1}: {s}_{\alpha} \in G\}$.
 For any $\alpha, \beta \in X$ with $\alpha < \beta$, ${F}_{\alpha}$ and ${F}_{\beta}$ are incompatible in ${\P}_{c}$.
 By shrinking $X$ to a smaller uncountable set if necessary, we may assume that there exist $R \in {\[{\kappa}^{+}\]}^{< \omega}$ and $n \in \omega \setminus \{0\}$ such that:
 \begin{enumerate}
  \item[(4)]
  $\forall \alpha, \beta \in X \[\alpha < \beta \implies {F}_{\alpha} \cap {F}_{\beta} = R\]$;
  \item[(5)]
  $\forall \alpha \in X \[\lc{F}_{\alpha} \setminus R\rc = n\]$;
  \item[(6)]
  $\forall \alpha, \beta \in X \[\alpha < \beta \implies \max({F}_{\alpha} \setminus R) < \min({F}_{\beta} \setminus R)\]$.
 \end{enumerate}
 Also, for any $\alpha, \beta \in X$ with $\alpha < \beta$, the incompatibility of ${F}_{\alpha}$ and ${F}_{\beta}$ implies that $\exists \zeta \in {F}_{\alpha} \setminus R \exists \xi \in {F}_{\beta} \setminus R \[c(\{\zeta, \xi\}) = 1\]$.
 Let $Y \subset X$ be a set with $\otp(Y) = \omega + \omega$.
 As $\Sa$ does not add any new countable sets of ordinals, $Y \in \V$.
 Since $R \in \V$, the sequence $\langle {F}_{\alpha} \setminus R: \alpha \in Y\rangle \in \V$, and it fulfills in $\V$ the hypotheses of Lemma \ref{lem:ultrafilter}, thus contradicting the supposition that $c$ witnesses ${\kappa}^{+} \not\rightarrow {\left({\kappa}^{+}, (\omega:2)\right)}^{2}$ in $\V$. 
 \end{proof}
 To complete the proof, note that by Claim \ref{claim:diff1}, for each $\alpha < {\kappa}^{+}$, the set ${D}_{\alpha} = \{F \in {\P}_{c}: \exists \beta \in F\[\alpha < \beta\]\}$ is dense in ${\P}_{c}$.
 If $G \subset {\P}_{c}$ is a filter such that $\forall \alpha < {\kappa}^{+}\[G \cap {D}_{\alpha} \neq 0\]$, then $X = \bigcup G \in {\[{\kappa}^{+}\]}^{{\kappa}^{+}}$ and ${\[X\]}^{2} \subset {K}_{0}$, a final contradiction.
\end{proof}
 \begin{Cor} \label{cor:difference}
  Let $\kappa \geq \omega$ be cardinal.
  It is consistent to have ${\kappa}^{+} \rightarrow {\left({\kappa}^{+}, (\omega:2)\right)}^{2}$, while ${\kappa}^{+} \not\rightarrow {({\kappa}^{+}, \omega + 2)}^{2}$.
 \end{Cor}
 \begin{proof}
  It needs to be seen that the hypotheses of Theorem \ref{thm:coloring} are consistent.
 Let $\V$ be a ground model satisfying $\GCH$, and suppose that in $\V$ $\Sa$ is an ${\omega}_{1}$-Suslin tree and that $\Ta$ is a ${\kappa}^{+}$-Suslin tree. 
  Using a bookkeeping device to ensure that all names for posets of size at most ${\kappa}^{+}$ are eventually considered, build a FS iteration of c.\@c.\@c.\@ posets $\langle {\P}_{\alpha}, {\mathring{\Q}}_{\alpha}: \alpha \leq {\kappa}^{++}\rangle$ as follows.
  At a stage $\alpha < {\kappa}^{++}$, if the object handed by the bookkeeping device is a full ${\P}_{\alpha}$-name $\mathring{\Q}$ for a poset such that ${\forces}_{\alpha} \; ``\mathring{\Q} \times \Sa \ \text{is c.\@c.\@c.\@}''$, then let ${\mathring{\Q}}_{\alpha} = \mathring{\Q}$.
  Otherwise let ${\mathring{\Q}}_{\alpha}$ be a full ${\P}_{\alpha}$-name for the trivial poset.
  Let $G$ be a $(\V, {\P}_{{\kappa}^{++}})$-generic filter.
  It is well-known that $\Sa$ remains an ${\omega}_{1}$-Suslin tree in $\V\[G\]$.
  Standard arguments show that ${2}^{\omega} = {\kappa}^{++}$ and $\MA(\Sa)$ hold in $\V\[G\]$.
  Thus if ${\kappa}^{+} = {\omega}_{1}$, then all the required statements hold in $\V\[G\]$.
  So assume that ${\kappa}^{+} > {\omega}_{1}$.
  It suffices to check that $\Ta$ remains a ${\kappa}^{+}$-Suslin tree in $\V\[G\]$.
  We check that $\Ta \times {\P}_{{\kappa}^{++}}$ is ${\kappa}^{+}$-c.\@c.\@ in $\V$.
  Suppose for a contradiction that $\langle \langle {s}_{\alpha}, {p}_{\alpha} \rangle: \alpha < {\kappa}^{+}\rangle \in \V$ is a sequence of pairwise incompatible elements of $\Ta \times {\P}_{{\kappa}^{++}}$.
  As $\Ta$ is ${\kappa}^{+}$-c.\@c.\@ in $\V$, there is a $(\V, \Ta)$-generic filter $H$ such that $\{\alpha < {\kappa}^{+}: {s}_{\alpha} \in H\}$ is cofinal in ${\kappa}^{+}$.
  Using the fact that $\Ta$ is a ${\kappa}^{+}$-Suslin tree in $\V$ and the fact that ${\omega}_{1} < {\kappa}^{+}$, find a $Y \in \V$ such that the following hold in $\V$: $Y \subset {\kappa}^{+}$, $\otp(Y) = {\omega}_{1}$ and $\forall \alpha, \beta \in Y \[\alpha < \beta \implies {s}_{\alpha} \; \not\perp \; {s}_{\beta}\]$.
  However, this implies that $\seq{p}{\alpha}{\in}{Y}$ is a sequence of pairwise incompatible elements of ${\P}_{{\kappa}^{++}}$, which contradicts the fact that ${\P}_{{\kappa}^{++}}$ is c.\@c.\@c.\@ in $\V$.
  This completes the proof that $\Ta$ remains a ${\kappa}^{+}$-Suslin tree in $\V\[G\]$.
 \end{proof}
 As mentioned in the introduction, Todorcevic~\cite{realpartition} distinguished between the relations $\c \rightarrow {\left( \c, \omega + 2 \right)}^{2}$ and $\c \rightarrow {\left( \c, (\omega : 2) \right)}^{2}$ using a measurable cardinal.
 Corollary \ref{cor:difference} does not seem to address this.
 \begin{Question} \label{q:partition5}
  Is $\ZFC + \c \rightarrow {\left( \c, (\omega : 2) \right)}^{2} + \c \not\rightarrow {\left( \c, \omega + 2 \right)}^{2}$ consistent relative to just $\ZFC$?
 \end{Question}
\section{Consistency of $\b \rightarrow {\left( \b, \alpha \right)}^{2}$ from a measurable cardinal} \label{sec:measurable}
We give a positive answer to Question \ref{q:partition2} relative to a measurable cardinal in this section.
The analysis of certain strong chain conditions will play a key role.
\begin{Def} \label{def:strongcc}
 We define the following strong forms of the countable chain condition on a poset:
 \begin{enumerate}
  \item
  we say that a poset $\P$ is \emph{$\sigma$-linked} if there exists a sequence $\seq{\P}{n}{\in}{\omega}$ such that $\P = {\bigcup}_{n \in \omega} {\P}_{n}$ and $\forall n \in \omega \forall p, p' \in {\P}_{n}\[p \; {\not\perp}_{\P} \; p'\]$.
  \item
   we say that a poset $\P$ has the \emph{$\sigma$-bounded-c.\@c.\@} if there exist a sequence $\seq{\P}{n}{\in}{\omega}$ and a function $f \in \BS$ such that:
 \begin{enumerate}
  \item[(a)]
  $\P = {\bigcup}_{n \in \omega}{{\P}_{n}}$;
  \item[(b)]
  for each $n \in \omega$, if $\seq{p}{i}{<}{f(n) + 1}$ is a sequence of elements of ${\P}_{n}$, then $\exists i < j < f(n) + 1$ such that ${p}_{i} \; {\not\perp}_{\P} \; {p}_{j}$.
 \end{enumerate}
  \item
  we say that a poset $\P$ has the \emph{$\sigma$-finite-c.\@c.\@} if there exists a sequence $\seq{\P}{n}{\in}{\omega}$ such that $\P = {\bigcup}_{n \in \omega}{{\P}_{n}}$ and for each $n \in \omega$, if $\seq{p}{i}{<}{\omega}$ is a sequence of elements of ${\P}_{n}$, then $\exists i < j < \omega\[{p}_{i} \; {\not\perp}_{\P} \; {p}_{j}\]$.   
 \end{enumerate}
\end{Def}
It is clear that $(1) \implies (2) \implies (3)$ and that $(3)$ implies that $\P$ is c.\@c.\@c.
It is known that none of these implications reverse.
\begin{Def} \label{def:knast}
 A poset $\P$ is said to have the \emph{$\kappa$-Knaster} property if for any sequence $\seq{p}{\alpha}{<}{\kappa} \subset \P$, there exists $A \in {\[\kappa\]}^{\kappa}$ such that ${p}_{\alpha} \; {\not\perp}_{\P} \; {p}_{\beta}$, for all $\alpha, \beta \in A$.
\end{Def}
\begin{Lemma} \label{lem:knast}
If $\P$ has the $\sigma$-finite-c.\@c.\@, then $\P$ has the $\kappa$-Knaster property for all $\kappa$ with $\cf(\kappa) > \omega$. 
\end{Lemma}
\begin{proof}
Suppose $\seq{\P}{n}{\in}{\omega}$ witnesses that $\P$ has the $\sigma$-finite.c.\@c.
Let $\seq{p}{\alpha}{<}{\kappa} \subset \P$ be any sequence.
Since $\cf(\kappa) > \omega$, there exist $n \in \omega$ and $B \in {\[\kappa\]}^{\kappa}$ such that ${p}_{\alpha} \in {\P}_{n}$, for all $\alpha \in B$.
Define $c: {\[B\]}^{2} \rightarrow 2$ by the condition that $c(\{\alpha, \beta\}) = 0$ iff ${p}_{\alpha} \; {\not\perp}_{\P} \; {p}_{\beta}$, for every $\alpha, \beta \in B$ with $\alpha < \beta$.
There cannot be a $1$-homogeneous set of order-type $\omega$.
So by the partition relation $\kappa \rightarrow {\left( \kappa, \omega \right)}^{2}$, there is a $0$-homogeneous set $A \in {\[B\]}^{\kappa}$.
Clearly, $\seq{p}{\alpha}{\in}{A}$ is as needed.
\end{proof}
\begin{Def} \label{def:unifandnormal}
For a limit ordinal $\delta$, a proper ideal $\I$ on $\delta$ is said to be \emph{uniform} if for all $\alpha < \delta$, $\alpha \in \I$.
A uniform ideal $\I$ on $\delta$ is called \emph{normal} if for any $X \in {\I}^{+}$ and any regressive map $f: X \rightarrow \delta$, there exists $\alpha \in \delta$ such that ${f}^{-1}(\{\alpha\}) \in {\I}^{+}$.
\end{Def}
\begin{Def} \label{def:fubinipower}
 Let $\delta$ be a limit ordinal and $\II$ a uniform ideal on $\delta$.
 We define ideals ${\II}^{n}$ on ${\[\delta\]}^{n}$ by induction on $1 \leq n < \omega$.
 By identifying ${\[\delta\]}^{1}$ with $\delta$, we define ${\II}^{1} = \II$.
 Now given ${\II}^{n}$ on ${\[\delta\]}^{n}$, where $1 \leq n < \omega$, we define ${\II}^{n + 1}$ to be
 \begin{align*}
  \left\{ X \subset {\[\delta\]}^{n + 1}: \left\{ \alpha < \delta: {(X)}_{\alpha} \in {\left( {\II}^{n} \right)}^{+} \right\} \in \II \right\}, 
 \end{align*}
 where ${(X)}_{\alpha} = \{F \setminus \{\alpha\}: F \in X \wedge \min(F) = \alpha \}$.
\end{Def}
It is easily checked that ${\II}^{n}$ is an ideal on ${\[\delta\]}^{n}$.
Moreover for $1 \leq n < \omega$ and $X \subset {\[\delta\]}^{n + 1}$, $X \in {\left( {\II}^{n + 1} \right)}^{+}$ iff $\left\{ \alpha < \delta: {(X)}_{\alpha} \in {\left( {\II}^{n} \right)}^{+}\right\} \in {\II}^{+}$, and because of the assumption that $\II$ is uniform, $X \in {\left( {\II}^{n + 1} \right)}^{\ast}$ iff $\left\{\alpha < \delta: {(X)}_{\alpha} \in {\left( {\II}^{n} \right)}^{\ast}\right\} \in {\II}^{\ast}$.
Lemmas \ref{lem:normal} and \ref{lem:groundcovering} are folklore.
We include a proof for the reader's convenience.
\begin{Lemma} \label{lem:normal}
 Suppose $\kappa$ is an uncountable cardinal and $\I$ is a normal ideal on $\kappa$.
 If $\P$ is any c.\@c.\@c.\@ forcing, then $\forces``\; \I \ \text{generates a normal ideal on} \ \kappa\;''$.
\end{Lemma}
\begin{proof}
If $G$ is $(\V, \P)$-generic, then in $\VG$, $\II = \{X \subset \kappa: \exists Y \in \I\[X \subset Y\]\}$ is the ideal generated by $\I$;
it is a uniform ideal on $\kappa$.
 Let $\mathring{\II}$ be a $\P$-name in $\V$ for the ideal generated by $\I$.
 Fix $\mathring{X} \in \VP$ and $p \in \P$.
 \begin{Claim} \label{claim:normal1}
  If $p \forces \mathring{X} \subset \kappa \wedge \mathring{X} \in \mathring{\II}$, then there is $Y \in \I$ such that $p \forces \mathring{X} \subset Y$.
 \end{Claim}
 \begin{proof}
 Since $p \forces \exists Y \in \I\[\mathring{X} \subset Y\]$ and since $\P$ is c.\@c.\@c.\@, we may find a maximal antichain $\{{p}_{n}: n \in \omega\}$ below $p$ and a set $\{{Y}_{n}: n \in \omega\} \subset \I$ such that for each $n \in \omega$, ${p}_{n} \forces \mathring{X} \subset {Y}_{n}$.
 By normality, $Y = \bigcup\{{Y}_{n}: n \in \omega\} \in \I$.
 Now for any $p' \leq p$, there exist $p'' \leq p'$ and $n \in \omega$ such that $p'' \forces \mathring{X} \subset {Y}_{n}$, whence $p'' \forces \mathring{X} \subset Y$.
 Thus $p \forces \mathring{X} \subset Y$.
 \end{proof}
 Now suppose $\mathring{f} \in \VP$ and $p \forces``\mathring{X} \subset \kappa \ \text{and} \ \mathring{f} \ \text{is a regressive map from} \ \mathring{X} \ \text{to} \ \kappa''$.
 For each $\alpha < \kappa$, let ${\mathring{X}}_{\alpha}$ be a $\P$-name such that $p \forces {\mathring{X}}_{\alpha} = {\mathring{f}}^{-1}(\{\alpha\})$.
 Assume that for all $\alpha < \kappa$, $p \forces {\mathring{X}}_{\alpha} \in \mathring{\II}$.
 Applying the claim find ${Y}_{\alpha} \in \I$ such that $p \forces {\mathring{X}}_{\alpha} \subset {Y}_{\alpha}$.
 By normality, $Y = \{0\} \cup \{\beta < \kappa: \exists \alpha < \beta \[\beta \in {Y}_{\alpha}\]\} \in \I$.
 Now $p \forces \mathring{X} \subset Y$.
 For otherwise, there would be $p' \leq p$, $\beta \in \kappa \setminus Y$, and $\alpha \in \kappa$ such that $p' \forces \beta \in \mathring{X} \wedge \mathring{f}(\beta) = \alpha$.
 It would follow that $\beta \in {Y}_{\alpha}$, and that since $\beta > 0$, $\alpha < \beta$.
 However this would imply that $\beta \in Y$, which is a contradiction. 
 Thus we conclude $p \forces \mathring{X} \subset Y$, and so $p \forces \mathring{X} \in \mathring{\II}$.
\end{proof}
\begin{Lemma} \label{lem:groundcovering}
 Suppose $\kappa$ is an uncountable cardinal and $\I$ is a normal ideal on $\kappa$.
 Let $\P$ be a c.\@c.\@c.\@ poset.
 Let $\mathring{\II}$ be a $\P$-name for the ideal generated by $\I$.
 Fix $1 \leq n < \omega$, $\mathring{X} \in \VP$, and $p \in \P$.
 If $p \forces \mathring{X} \subset {\[\kappa\]}^{n} \wedge \mathring{X} \in {\mathring{\II}}^{n}$, then there exists $Y \in {\I}^{n}$ such that $p \forces \mathring{X} \subset Y$.
\end{Lemma}
\begin{proof}
 By induction on $n$.
 When $n = 1$, Claim \ref{claim:normal1} gives what is needed.
 Assume the statement is true for $n$ and consider $n + 1$.
 By the maximal principle, there is $\mathring{A} \in \VP$ such that $p \forces \left\{\alpha < \kappa: {\left(\mathring{X}\right)}_{\alpha} \in {\left( {\mathring{\II}}^{n} \right)}^{+}\right\} = \mathring{A} \in \mathring{\II}$.
 Applying the case $n = 1$, which has already been proved, find $B \in \I$ such that $p \forces \mathring{A} \subset B$.
 We are going to define a sequence $\seq{D}{\alpha}{\in}{\kappa}$ as follows.
 First consider the case when $\alpha \in \kappa \setminus B$.
 Applying the maximal principle again, there is a ${\mathring{C}}_{\alpha} \in \VP$ such that $p \forces \left\{ F \setminus \{\alpha\}: F \in \mathring{X} \wedge \min(F) = \alpha\right\} = {\mathring{C}}_{\alpha} \in {\mathring{\II}}^{n}$.
 Once again there is ${D}_{\alpha} \in {\I}^{n}$ such that $p \forces {\mathring{C}}_{\alpha} \subset {D}_{\alpha}$.
 Next if $\alpha \in B$, then set ${D}_{\alpha} = {\[\kappa\]}^{n}$.
 Now define $Y = \{s \in {\[\kappa\]}^{n + 1}: s\setminus\{\min(s)\} \in {D}_{\min(s)}\}$.
 Then for any $\alpha < \kappa$, ${(Y)}_{\alpha} \subset {D}_{\alpha}$, and so $\{\alpha < \kappa: {(Y)}_{\alpha} \in {\left( {\I}^{n} \right)}^{+}\} \subset B$.
 Therefore, $Y \in {\I}^{n + 1}$.
 It is not hard to see that $p \forces \mathring{X} \subset Y$.
 For suppose not.
 Let $G$ be a $(\V, \P)$-generic filter with $p \in G$ such that in $\VG$, there exists $F \in \mathring{X}\[G\]$ with $F \notin Y$.
 Let $\alpha = \min(F)$.
 If $\alpha \in B$, then $F \setminus \{\alpha\} \in {\[\kappa\]}^{n} = {D}_{\alpha}$, whence $F \in Y$.
 On the other hand if $\alpha \in \kappa \setminus B$, then $F \setminus \{\alpha\} \in {\mathring{C}}_{\alpha}\[G\] \subset {D}_{\alpha}$, and so once again $F \in Y$.
 Thus both cases lead to a contradiction.
\end{proof}
\begin{Lemma} \label{lem:sigmaboundedquotients}
 Suppose $\kappa$ is a measurable cardinal and $\U$ is a normal measure on $\kappa$.
 Let $\I = {\U}^{\ast}$.
 Suppose $\langle {\P}_{\alpha}, {\mathring{\Q}}_{\alpha}: \alpha \leq \kappa \rangle$ is an FS iteration such that for each $\alpha < \kappa$, ${\forces}_{\alpha} ``\;{\mathring{\Q}}_{\alpha} \ \text{is} \ \sigma\text{-linked and} \ \lc {\mathring{\Q}}_{\alpha} \rc < \kappa\;''$.
 Let $G$ be $(\V, {\P}_{\kappa})$-generic.
 In $\VG$ define $\II$ to be the ideal on $\kappa$ generated by $\I$.
 Then for each $1 \leq n < \omega$, ${\left( {\II}^{n} \right)}^{+}$ has the $\sigma$-bounded-c.\@c.
\end{Lemma}
\begin{proof}
 Because the ${\P}_{\alpha}$ are c.\@c.\@c.\@ and because of the hypothesis that for all $\alpha < \kappa$  ${\forces}_{\alpha}``{\mathring{\Q}}_{\alpha} \ \text{is} \ \sigma\text{-linked and} \ \lc {\mathring{\Q}}_{\alpha} \rc < \kappa''$, we can find a sequence $\seq{\mathring{g}}{\alpha}{<}{\kappa}$ and a dense set $D \subset {\P}_{\kappa}$ such that:
 \begin{enumerate}
 \item
 for each $\alpha < \kappa$, ${\mathring{g}}_{\alpha}$ is a ${\P}_{\alpha}$-name, ${\forces}_{\alpha}\; {\mathring{g}}_{\alpha}: {\mathring{\Q}}_{\alpha} \rightarrow \omega$ and ${\forces}_{\alpha} \forall q, q' \in {\mathring{\Q}}_{\alpha}\[{\mathring{g}}_{\alpha}(q) = {\mathring{g}}_{\alpha}(q') \implies q \; {\not\perp}_{{\mathring{\Q}}_{\alpha}} \; q'\]$;
  \item
  for each $p \in D$, there exists a function ${\sigma}_{p}: \suppt(p) \rightarrow \omega$ such that $\forall \alpha \in \suppt(p)\[p\restrict\alpha \;{\forces}_{\alpha} \;{\mathring{g}}_{\alpha}(p(\alpha)) = {\sigma}_{p}(\alpha)\]$;
  \item
  for each $\alpha < \kappa$, $\lc \{p\restrict\alpha: p \in D \}\rc < \kappa$.
 \end{enumerate}
 For any $p \in D$, let ${m}_{p} = \otp(\suppt(p))$, and for each $i < {m}_{p}$ let $\suppt(p)(i)$ denote the $i$th element of $\suppt(p)$.
 Define a function ${\tau}_{p}: {m}_{p} \rightarrow \omega$ by ${\tau}_{p}(i) = {\sigma}_{p}(\suppt(p)(i))$, for each $i < {m}_{p}$.
 Let $G$ be $(\V, {\P}_{\kappa})$-generic.
 In $\VG$, $\II$ is a normal ideal because of Lemma \ref{lem:normal}; hence it is $\kappa$-complete.
 Now a simple induction on $1 \leq n < \omega$ shows that ${\II}^{n}$ is $\kappa$-complete.
 Returning to $\V$, for each $1 \leq n < \omega$, let ${\mathring{\LL}}_{n}$ be a full ${\P}_{\kappa}$-name denoting ${\left( {\II}^{n} \right)}^{+}$, and for each $\mathring{X} \in \dom({\mathring{\LL}}_{n})$ and $s \in {\[\kappa\]}^{n}$, let $R(s, \mathring{X}) = \{p \in D: p \forces s \in \mathring{X}\}$.
 Move back to $\VG$ and work there.
 We first point out a simple consequence of (3) and the $\kappa$-completeness of ${\II}^{n}$.
 If $X \in {\left( {\II}^{n} \right)}^{+}$ and $\alpha < \kappa$, then for any $\mathring{X} \in \dom({\mathring{\LL}}_{n})$ such that $\mathring{X}\[G\] = X$, there exist $q \in {\P}_{\alpha}$, $m \in \omega$, and $\tau \in {\omega}^{m}$ such that $\{s \in X: \exists p \in G \cap R(s, \mathring{X})\[p \restrict \alpha = q \wedge {m}_{p} = m \wedge {\tau}_{p} = \tau\]\} \in {\left( {\II}^{n} \right)}^{+}$.
 For each $m \in \omega$ and $\tau \in {\omega}^{m}$, define ${\F}_{m, \tau}$ by stipulating that $X \in {\F}_{m, \tau}$ iff $X \in {\left( {\II}^{n} \right)}^{+}$ and there exists $\mathring{X} \in \dom({\mathring{\LL}}_{n})$ such that 
 \begin{align*}
  \mathring{X}\[G\] = X \wedge \left\{s \in X: \exists p \in G \cap R(s, \mathring{X})\[{m}_{p} = m \wedge {\tau}_{p} = \tau \]\right\} \in {\left( {\II}^{n} \right)}^{+}.
 \end{align*}
 By the above remarks ${\left( {\II}^{n} \right)}^{+} = \bigcup\{{\F}_{m, \tau}: m \in \omega \wedge \tau \in {\omega}^{m}\}$.
 For $m, \tau$ with $m \in \omega$ and $\tau \in {\omega}^{m}$, let $a = {2}^{m}$ and $b = m + 1$.
 Define $f(m, \tau) =$
 \begin{align*}
  b!{a}^{b + 1}\left( 1 - \displaystyle\frac{1}{2!a} - \displaystyle\frac{2}{3!{a}^{2}} - \dotsb - \displaystyle\frac{b - 1}{b!{a}^{b - 1}}\right).
 \end{align*}
 We claim that $\langle {\F}_{m, \tau}: m \in \omega \wedge \tau \in {\omega}^{m}\rangle$ and $f$ witness that ${\left( {\II}^{n} \right)}^{+}$ has the $\sigma$-bounded-c.\@c.
 Suppose not.
 Fix $m \in \omega$ and $\tau \in {\omega}^{m}$.
 Let $\{{X}_{j}: j < f(m, \tau) + 1\} \subset {\F}_{m, \tau}$ and suppose that for all $j < j' < f(m, \tau) + 1$, ${X}_{j} \cap {X}_{j'} \in {\II}^{n}$.
 For each $j < f(m, \tau) + 1$, fix ${\mathring{X}}_{j} \in \dom({\mathring{\LL}}_{n})$ such that
 \begin{align*}
  {\mathring{X}}_{j}\[G\] = {X}_{j} \wedge \{s \in {X}_{j}: \exists p \in G \cap R(s, {\mathring{X}}_{j})\[{m}_{p} = m \wedge {\tau}_{p} = \tau\]\} \in {\left( {\II}^{n} \right)}^{+}.
 \end{align*}
 Applying Lemma \ref{lem:groundcovering}, find $q \in G$ and $Y \in {\I}^{n}$ such that for each $j < j' < f(m, \tau) + 1$, the relation $q \forces {\mathring{X}}_{j} \cap {\mathring{X}}_{j'} \subset Y$ holds in $\V$.
 Put $\alpha = \sup\{\xi + 1: \xi \in \suppt(q)\} < \kappa$.
 Because of (3) above, for each $j < f(m, \tau) + 1$, there exists ${r}_{j} \in {\P}_{\alpha}$ such that ${Y}_{j} = \{s \in {X}_{j}: \exists p \in G \cap R(s, {\mathring{X}}_{j})\[p \restrict \alpha = {r}_{j} \wedge {m}_{p} = m \wedge {\tau}_{p} = \tau\]\} \in {\left( {\II}^{n} \right)}^{+}$.
 Choose ${s}_{j} \in {Y}_{j}$ and ${p}^{\ast}_{j} \in G$ witnessing this.
 As $G$ is a filter, there is ${q}_{0} \in G$ such that ${q}_{0} \leq q$ and ${q}_{0}\restrict\alpha \leq {r}_{j}$, for each $j < f(m, \tau) + 1$.
 Back in $\V$, let ${Z}_{j} = \{s \in {\[\kappa\]}^{n}: \exists p \in R(s, {\mathring{X}}_{j})\[p\restrict\alpha = {r}_{j} \wedge {m}_{p} = m \wedge {\tau}_{p} = \tau\]\}$.
 Note that ${\I}^{n}$ is a maximal ideal in $\V$ and that ${\I}^{n} = {\Pset}^{\V}\left( {\[\kappa\]}^{n} \right) \cap {\II}^{n}$.
 It follows that ${Z}_{j} \in {\left( {\I}^{n} \right)}^{\ast}$.
 Let $Z = {\bigcap}_{j < f(m, \tau) + 1}{{Z}_{j}}$.
 Then $Z \setminus Y \neq 0$.
 Choose $s \in Z \setminus Y$ and for each $j < f(m, \tau) + 1$, choose ${p}_{j}$ witnessing that $s \in {Z}_{j}$.
 Thus for all $j < f(m, \tau) + 1$, $\lc \suppt({p}_{j}) \setminus \alpha \rc \leq m + 1$.
 By Theorem III of \cite{findeltasystem}, there exist $F \subset f(m, \tau) + 1$ and $S$ such that $\lc F \rc > {2}^{m}$ and $\forall j, j' \in F\[j < j' \implies \left( \suppt({p}_{j})\setminus \alpha \right) \cap \left( \suppt({p}_{j'})\setminus\alpha\right) = S\]$.
 For each $j \in F$, define ${x}_{j} = \{i < m: \suppt({p}_{j})(i) \in S\}$.
 As $\lc F \rc > {2}^{m}$, we can find $j, j' \in F$ with $j < j'$ and $x \in \Pset(m)$ such that ${x}_{j} = {x}_{j'} = x$.
 We are going to find $p \in {\P}_{\kappa}$ such that $p \leq q, {p}_{j}, {p}_{j'}$.
 This would be a contradiction because in $\V$, $p \forces s \in {\mathring{X}}_{j} \cap {\mathring{X}}_{j'} \subset Y$, and yet $s \in Z \setminus Y$.
 
 Define $p(\beta) = {q}_{0}(\beta)$, for all $\beta < \alpha$.
 Then $p \restrict \alpha \leq q\restrict\alpha, {p}_{j}\restrict\alpha, {p}_{j'}\restrict\alpha$.
 If $\beta \in \suppt({p}_{j}) \setminus \left( S \cup \alpha \right)$, then define $p(\beta) = {p}_{j}(\beta)$, and if $\beta \in \suppt({p}_{j'}) \setminus \left( S \cup \alpha \right)$, then define $p(\beta) = {p}_{j'}(\beta)$.
 Finally if $\beta \in \kappa\setminus\left( \alpha \cup \suppt({p}_{j}) \cup \suppt({p}_{j'})\right)$, then define $p(\beta) = {\mathring{\mathbbm{1}}}_{\beta}$.
 If $S = 0$, then the definition of $p$ is complete and $p$ is as required.
 Suppose $S \neq 0$.
 Note that $S = \{\suppt({p}_{j})(i): i \in x \} = \{\suppt({p}_{j'})(i): i \in x\}$ and that $\forall i \in x\[\suppt({p}_{j})(i) = \suppt({p}_{j'})(i)\]$.
 For ease of notation, denote $\suppt({p}_{j})(i)$ by ${\gamma}_{i}$, for each $i \in x$.
 We define $p({\gamma}_{i})$ by induction on $i \in x$.
 Fix $i \in x$ and suppose that $p\restrict {\gamma}_{i}$ has been defined in such a way that $p\restrict {\gamma}_{i}\leq {p}_{j}\restrict{\gamma}_{i}, {p}_{j'}\restrict{\gamma}_{i}$.
 Then $p\restrict{\gamma}_{i} \; {\forces}_{{\gamma}_{i}} \; {\mathring{g}}_{{\gamma}_{i}}({p}_{j}({\gamma}_{i})) = {\mathring{g}}_{{\gamma}_{i}}({p}_{j'}({\gamma}_{i}))$, and so there is $\mathring{q} \in \dom({\mathring{\Q}}_{{\gamma}_{i}})$ such that $p\restrict{\gamma}_{i}\; {\forces}_{{\gamma}_{i}} \; \mathring{q} \leq {p}_{j}({\gamma}_{i}), {p}_{j'}({\gamma}_{i})$.
 Set $p({\gamma}_{i}) = \mathring{q}$.
 Now $p\restrict{\gamma}_{i} + 1$ is fully defined and $p\restrict {\gamma}_{i} + 1 \leq {p}_{j} \restrict {\gamma}_{i} + 1, {p}_{j'} \restrict {\gamma}_{i} + 1$.
 Moreover, if $x\setminus \left(i + 1\right) \neq 0$ and $i' = \min(x\setminus \left(i + 1\right))$, then $p\restrict{\gamma}_{i'}$ is fully defined and $ p\restrict {\gamma}_{i'} \leq {p}_{j} \restrict {\gamma}_{i'}, {p}_{j'} \restrict {\gamma}_{i'}$, so that the induction can proceed.
 This completes the definition of $p$ and it is easy to check that $p \leq q, {p}_{j}, {p}_{j'}$. 
 \end{proof}
 We next show that if $\II$ is any normal ideal on an uncountable cardinal $\kappa$ with the property that for each $1 \leq n < \omega$, ${\left( {\II}^{n} \right)}^{+}$ is $\kappa$-Knaster, then $\kappa \rightarrow {(\kappa, \alpha)}^{2}$ holds for all $\alpha < {\omega}_{1}$.
 Our Theorem \ref{thm:main3} should be compared with Kunen's classical result from \cite{kunreal} that if $\kappa$ is real valued measurable, then $\kappa \rightarrow {(\kappa, \alpha)}^{2}$ holds for all $\alpha < {\omega}_{1}$.
 \begin{Def} \label{def:K1}
  Let $\delta$ be any ordinal and $c: {\[\delta\]}^{2} \rightarrow 2$ be a coloring.
  For $\alpha \in \delta$, define ${K}_{1, c}(\alpha) = \{\beta < \delta: \alpha < \beta \wedge c(\{\alpha, \beta\}) = 1\}$.
  We will omit the subscript $c$ when it is clear from the context.
  For $X \subset \delta$ and $n \in \omega$, define ${X}^{\[n\]} = \{F \in {\[X\]}^{n}: {\[F\]}^{2} \subset {K}_{1, c}\}$.
 \end{Def}
 \begin{Lemma} \label{lem:manyK1}
  Let $\kappa > \omega$ be a cardinal and let $\II$ be a normal ideal on $\kappa$.
  Suppose $c: {\[\kappa\]}^{2} \rightarrow 2$ is a coloring with no $0$-homogeneous set of size $\kappa$.
  Then for every $X \in {\II}^{+}$ there exists $\alpha \in X$ such that $X \cap {K}_{1}(\alpha) \in {\II}^{+}$.
 \end{Lemma}
 \begin{proof}
  Suppose not.
  We get a contradiction to the hypotheses by constructing a $0$-homogeneous subset of $X$ of size $\kappa$ as follows.
  Fix $\delta < \kappa$ and suppose $\{{\alpha}_{\gamma}: \gamma < \delta\} \subset X$ is given.
  Assume that for all $\beta < \gamma < \delta$, ${\alpha}_{\beta} < {\alpha}_{\gamma}$ and that $c(\{{\alpha}_{\beta}, {\alpha}_{\gamma}\}) = 0$.
  By the normality of $\II$, $Y = \left( {\bigcup}_{\gamma < \delta} {\left( X \cap {K}_{1}({\alpha}_{\gamma}) \right) } \right) \cup \left( {\bigcup}_{\gamma < \delta}{\left( {\alpha}_{\gamma} + 1\right)} \right) \in \II$.
  Choose ${\alpha}_{\delta} \in X \setminus Y$.
  It is clear that ${\alpha}_{\delta}$ is as needed.
 \end{proof}
 \begin{Lemma} \label{lem:moreK1}
  Let $\kappa$, $\II$, and $c$ satisfy the hypotheses of Lemma \ref{lem:manyK1}.
  Let $\theta$ be a large enough regular cardinal.
  Suppose $n \in \omega$ and suppose ${M}_{0} \in \dotsb \in {M}_{n + 1}$ is a chain such that for each $0 \leq i \leq n + 1$,
  \begin{enumerate}
   \item
   ${M}_{i} \prec H(\theta)$ with $\kappa, \II, c \in {M}_{i}$;
   \item
   $\lc {M}_{i} \rc < \kappa$ and ${M}_{i} \cap \kappa \in \kappa$.
  \end{enumerate}
  For each $0 \leq i \leq n + 1$, define ${I}_{i} = \bigcup\left( {M}_{i} \cap \II \right)$.
  If ${\alpha}_{0}, \dotsc, {\alpha}_{n}$ are members of $\kappa$ such that: 
  \begin{enumerate}
   \item[(3)]
   for each $0 \leq i \leq n$, ${\alpha}_{i} \in {M}_{i + 1} \setminus {I}_{i}$;
   \item[(4)]
   ${\[\{{\alpha}_{0}, \dotsc, {\alpha}_{n}\}\]}^{2} \subset {K}_{1}$,
  \end{enumerate}
  then ${K}_{1}({\alpha}_{0}) \cap \dotsb \cap {K}_{1}({\alpha}_{n}) \in {\II}^{+}$.
 \end{Lemma}
 \begin{proof}
  The proof is by induction on $n$.
  When $n = 0$, ${K}_{1}({\alpha}_{0}) \in {\II}^{+}$ because by Lemma \ref{lem:manyK1}, $\{\alpha < \kappa: {K}_{1}(\alpha) \in \II\} \in \II \cap {M}_{0}$.
  Suppose the claim holds for $n \in \omega$ and consider ${M}_{0} \in \dotsb \in {M}_{n + 1} \in {M}_{n + 2}$ and ${\alpha}_{0}, \dotsc, {\alpha}_{n}, {\alpha}_{n + 1}$ satisfying the hypotheses of the claim.
  By the induction hypothesis, $X = {K}_{1}({\alpha}_{0}) \cap \dotsb \cap {K}_{1}({\alpha}_{n}) \in {\II}^{+} \cap {M}_{n + 1}$.
  Again by Lemma \ref{lem:manyK1}, $Y = \{\alpha \in X: X \cap {K}_{1}(\alpha) \in \II\} \in \II \cap {M}_{n + 1}$.
  By (3) ${\alpha}_{n + 1} \notin Y$ and by (4) ${\alpha}_{n + 1} \in X$.
  Therefore $X \cap {K}_{1}({\alpha}_{n + 1}) = {K}_{1}({\alpha}_{0}) \cap \dotsb \cap {K}_{1}({\alpha}_{n}) \cap {K}_{1}({\alpha}_{n + 1}) \in {\II}^{+}$.
 \end{proof}
 \begin{Lemma} \label{lem:shifting}
  Let $\kappa$, $\II$, and $c$ satisfy the hypotheses of Lemma \ref{lem:manyK1}.
  Suppose in addition that for each $1 \leq n < \omega$, ${\left( {\II}^{n} \right)}^{+}$ has the $\kappa$-Knaster property.
  Fix $X \in {\II}^{+}$ and $n \in \omega$.
  For every $Z \in \Pset({X}^{\[n + 1\]}) \cap {\left( {\II}^{n + 1} \right)}^{+}$ there exists $\alpha \in X$ such that $\forall A \in {\II}^{n + 1} \exists F \in Z\setminus A\[F \in {\left( {K}_{1}(\alpha) \right)}^{\[n + 1 \]}\]$.
 \end{Lemma}
 \begin{proof}
  First if $n = 0$, then recalling our identification of ${\[\kappa\]}^{1}$ with $\kappa$, the hypotheses say that $Z \subset X$ and that $Z \in {\II}^{+}$.
  So by Lemma \ref{lem:manyK1}, there is $\alpha \in Z \subset X$ such that $Z \cap {K}_{1}(\alpha) \in {\II}^{+}$.
  Given any $A \in {\II}^{1} = \II$, choose any $\beta \in \left( Z \cap {K}_{1}(\alpha)\right)\setminus A$.
  Then $\beta \in Z\setminus A$ and $\beta \in {K}_{1}(\alpha)$, as required.
  
  Assume that $1 \leq n < \omega$.
  Aiming for a contradiction, fix a counterexample $Z \in \Pset({X}^{\[n + 1\]}) \cap {\left( {\II}^{n + 1} \right)}^{+}$.
  Then for each $\alpha \in X$ there exists ${A}_{\alpha} \in {\II}^{n + 1}$ such that for each $F \in Z \setminus {A}_{\alpha}$, $F \notin {\left( {K}_{1}(\alpha) \right)}^{\[n + 1\]}$.
  For each $\alpha \in X$ note that ${L}_{\alpha} = \left\{\beta < \kappa: {({A}_{\alpha})}_{\beta} \in {\left( {\II}^{n} \right)}^{+}\right\} \in \II$.
  For $\alpha \in \kappa\setminus X$, let ${L}_{\alpha} = 0$.
  By the normality of $\II$, $L = \{\beta < \kappa: \exists \alpha < \beta\[\beta \in {L}_{\alpha}\]\} \in \II$.
  Since $Z \in {\left( {\II}^{n + 1} \right)}^{+}$, by definition ${Z}^{\ast} = \{\beta < \kappa: {\left( Z \right)}_{\beta} \in {\left( {\II}^{n} \right)}^{+} \} \in {\II}^{+}$.
  Observe that ${Z}^{\ast} \subset X$ and that ${Z}^{\ast\ast} = {Z}^{\ast} \setminus L \in {\II}^{+}$.
  Now it is easy to see that for each $\beta \in {Z}^{\ast\ast}$, ${\left( Z \right)}_{\beta}\setminus {B}_{\beta} \in {\left( {\II}^{n} \right)}^{+}$, where ${B}_{\beta} = \bigcup\left\{{\left( {A}_{\alpha} \right)}_{\beta}: \alpha \in {Z}^{\ast\ast} \wedge \alpha < \beta \right\}$.
  Using the fact that ${\left( {\II}^{n} \right)}^{+}$ has the $\kappa$-Knaster property and the fact that ${Z}^{\ast\ast} \in {\[\kappa\]}^{\kappa}$, it is possible to find ${Z}^{\ast\ast\ast} \in {\[{Z}^{\ast\ast}\]}^{\kappa}$ such that $\forall \beta, \gamma \in {Z}^{\ast\ast\ast}\[\left( {\left( Z \right)}_{\beta} \setminus {B}_{\beta} \right) \cap \left( {\left( Z \right)}_{\gamma} \setminus {B}_{\gamma} \right) \in {\left( {\II}^{n} \right)}^{+}\]$.
  Since $c$ does not have any $0$-homogeneous sets of size $\kappa$, there exist $\beta, \gamma \in {Z}^{\ast\ast\ast}$ such that $\beta < \gamma$ and $c(\{\beta, \gamma\}) = 1$. 
  Choose $G \in \left( {\left( Z \right)}_{\beta} \setminus {B}_{\beta} \right) \cap \left( {\left( Z \right)}_{\gamma} \setminus {B}_{\gamma} \right)$.
  There exist $H, F \in Z$ such that $G = H \setminus \{\beta\} = F \setminus \{\gamma\}$, $\beta = \min(H)$, and $\gamma = \min(F)$.
  Thus $H, F \in {X}^{\[n + 1\]}$, and since $c(\{\beta, \gamma \}) = 1$, it follows that $F \in {\left( {K}_{1}(\beta) \right)}^{\[n + 1 \]}$.
  This implies $F \in {A}_{\beta}$ because of the choice of ${A}_{\beta}$.
  Therefore $G \in {\left( {A}_{\beta} \right)}_{\gamma}$.
  However since $\beta \in {Z}^{\ast\ast}$ and $\beta < \gamma$, we get that $G \in {B}_{\gamma}$, contradicting the choice of $G$.
 \end{proof}
 \begin{Lemma} \label{lem:goingdown}
  Let $\kappa, \II, c, \theta, n, {M}_{0}, \dotsc, {M}_{n + 1}$, and ${\alpha}_{0}, \dotsc, {\alpha}_{n}$ be as in Lemma \ref{lem:moreK1}.
  Moreover suppose that for all $1 \leq m < \omega$, ${\left( {\II}^{m} \right)}^{+}$ has the $\kappa$-Knaster property.
  Let $X \in {M}_{0} \cap {\II}^{+}$ and suppose that $H = \{{\alpha}_{0}, \dotsc, {\alpha}_{n}\} \subset X$.
  Then $\exists \alpha \in X \cap {M}_{0}\[H \in {\left( {K}_{1}(\alpha)\right)}^{\[n + 1\]}\]$.
 \end{Lemma}
 \begin{proof}
  Suppose not.
  For each $\beta < \kappa$ define $Y(\beta) =$
  \begin{align*}
   \left\{F \in {X}^{\[n + 1\]}: \forall \alpha \leq \beta \[\alpha \in X \implies F \notin {\left( {K}_{1} (\alpha) \right)}^{\[n + 1\]}\]\right\}.
  \end{align*}
  It is clear from the definition that $\forall \beta \leq \gamma < \kappa\[Y(\gamma) \subset Y(\beta)\]$.
  Note that $\langle Y(\beta): \beta < \kappa \rangle \in {M}_{0}$.
  If $\beta \in {M}_{0} \cap \kappa$, then $Y(\beta) \in {\left( {\II}^{n + 1} \right)}^{+}$ because $H \in Y(\beta)$ (by the suppose not).
  It follows from the elementarity of ${M}_{0}$ that $\forall \beta < \kappa\[Y(\beta) \in {\left( {\II}^{n + 1} \right)}^{+}\]$.
  Since ${\left( {\II}^{n + 1} \right)}^{+}$ has the $\kappa$-Knaster property, we can fix $\beta < \kappa$ such that $\forall \beta \leq \gamma < \kappa\[Y(\beta) \setminus Y(\gamma) \in {\II}^{n + 1}\]$.
  Let $Z = Y(\beta)$.
  Then $Z \in \Pset({X}^{\[n + 1\]}) \cap {\left( {\II}^{n + 1} \right)}^{+}$ and since $X \in {\II}^{+}$, Lemma \ref{lem:shifting} applies and implies that there is $\alpha \in X$ such that $\forall A \in {\II}^{n + 1} \exists F \in Z\setminus A \[F \in {\left( {K}_{1}(\alpha) \right)}^{\[n + 1\]}\]$.
  Let $\gamma = \max\{\alpha, \beta\}$ and $A = Y(\beta) \setminus Y(\gamma) \in {\II}^{n + 1}$.
  Find $F \in Z \setminus A$ with $F \in {\left( {K}_{1}(\alpha) \right)}^{\[n + 1\]}$.
  But now in contradiction to the definition of $Y(\gamma)$, we have $F \in Y(\gamma)$, $\alpha \leq \gamma$, $\alpha \in X$, and yet $F \in {\left( {K}_{1}(\alpha) \right)}^{\[n + 1\]}$.
 \end{proof}
 \begin{Theorem} \label{thm:main3}
  Let $\kappa > \omega$ be a cardinal.
  Let $\II$ be a normal ideal on $\kappa$.
  If for all $1 \leq n < \omega$, ${\left( {\II}^{n} \right)}^{+}$ has the $\kappa$-Knaster property, then $\kappa \rightarrow {\left( \kappa, \alpha \right)}^{2}$ for every $\alpha < {\omega}_{1}$. 
 \end{Theorem}
 \begin{proof}
  Let $c: {\[\kappa\]}^{2} \rightarrow 2$ be a coloring.
  Assume that there is no $0$-homogeneous set of size $\kappa$.
  Fix $\delta < {\omega}_{1}$.
  We must find a $1$-homogeneous set of order type $\delta$.
  In fact it will be notationally more convenient to find a $1$-homogeneous set of order type $\delta + 1$, which is of course sufficient.
  We may also assume that $\delta \geq \omega$ because otherwise we may use the partition relation $\kappa \rightarrow {\left( \kappa, \omega \right)}^{2}$.
  Fix a sequence $\seq{N}{\gamma}{<}{\delta + 2}$ such that:
  \begin{enumerate}
   \item
   for each $\gamma < \delta + 2$, ${N}_{\gamma} \prec H(\theta)$ with $\kappa, \II, c \in {N}_{\gamma}$;
   \item
   for each $\gamma < \delta + 2$, $\lc {N}_{\gamma} \rc < \kappa$ and ${N}_{\gamma} \cap \kappa \in \kappa$;
  \item
  $\forall \beta < \gamma < \delta + 2\[{N}_{\beta} \in {N}_{\gamma} \wedge {N}_{\beta} \subset {N}_{\gamma}\]$.
  \end{enumerate}
  For each $\gamma < \delta + 2$, let ${I}_{\gamma} = \bigcup\left( {N}_{\gamma} \cap \II \right)$.
  We will find a sequence $\seq{\alpha}{\gamma}{<}{\delta + 1}$ such that:
  \begin{enumerate}
   \item[(4)]
   for each $\gamma < \delta + 1$, ${\alpha}_{\gamma} \in \kappa \cap \left( {N}_{\gamma + 1} \setminus {I}_{\gamma} \right)$;
   \item[(5)]
   ${\[\{{\alpha}_{\gamma}: \gamma < \delta + 1\}\]}^{2} \subset {K}_{1, c}$.
  \end{enumerate}
  It is easy to see that $\{{\alpha}_{\gamma}: \gamma < \delta + 1\}$ would be a subset of $\kappa$ of order type $\delta + 1$ so that this would be the $1$-homogeneous set we are looking for.
  Let $\{{\gamma}_{n}: n \in \omega\}$ be an enumeration without repetition of $\delta + 1$ such that ${\gamma}_{0} = 0$ and ${\gamma}_{1} = \delta$.
  First applying Lemma \ref{lem:manyK1} to $\kappa \setminus {I}_{0} \in {\II}^{+}$, we can find ${\alpha}_{0} \in {N}_{1}$ such that ${\alpha}_{0} \in \kappa \setminus {I}_{0}$ and $\left( \kappa \setminus {I}_{0} \right) \cap {K}_{1}({\alpha}_{0}) \in {\II}^{+}$.
  Next to find ${\alpha}_{\delta}$, let $X = \left( \kappa \setminus {I}_{0} \right) \cap {K}_{1}({\alpha}_{0})$ and note that $X \in {\II}^{+}$.
  So $X \setminus {I}_{\delta} \neq 0$ because ${I}_{\delta} \in \II$.
  Since $X\setminus{I}_{\delta} \in {N}_{\delta + 1}$, there exists ${\alpha}_{\delta} \in {N}_{\delta + 1} \cap \left( X \setminus {I}_{\delta} \right)$.
  It is clear that ${\alpha}_{\delta}$ is as needed.
  
  Now fix $1 < n < \omega$ and assume that for all $i < n$, ${\alpha}_{{\gamma}_{i}}$ has been defined so that (4) and (5) are satisfied.
  Consider $L = \{i < n: {\gamma}_{i} < {\gamma}_{n}\}$.
  Note that $0 \in L$.
  Let ${\beta}_{0} < \dotsb < {\beta}_{l - 1}$ enumerate the set $\{{\gamma}_{i}: i \in L\}$ in increasing order.
  Define ${\beta}_{l} = {\gamma}_{n}$.
  Then ${\beta}_{0} < \dotsb < {\beta}_{l - 1} < {\beta}_{l}$ holds.
  Now we can apply Lemma \ref{lem:moreK1} with ${M}_{i} = {N}_{{\beta}_{i}}$, for all $0 \leq i \leq l$, and ${\alpha}_{i} = {\alpha}_{{\beta}_{i}}$, for all $0 \leq i \leq l - 1$ to conclude that $Y = \bigcap\{{K}_{1}({\alpha}_{{\gamma}_{i}}): i \in L\} \in {\II}^{+}$.
  Redefine $X = Y\setminus{I}_{{\gamma}_{n}}$ and note that $X \in {N}_{\left( {\gamma}_{n} \right) + 1} \cap {\II}^{+}$ because ${I}_{{\gamma}_{n}} \in \II$.
  Consider $M = \{i < n: {\gamma}_{i} > {\gamma}_{n}\}$.
  Note $1 \in M$.
  Let ${\eta}_{1} < \dotsb < {\eta}_{m}$ enumerate the set $\{{\gamma}_{i}: i \in M\}$ in increasing order.
  For $0 \leq i \leq m$ define ${\beta}_{i}$ as follows: define ${\beta}_{0} = \left( {\gamma}_{n} \right) + 1$; if $i \neq 0$, then put ${\beta}_{i} = \left( {\eta}_{i} \right) + 1$.
  Then ${\beta}_{0} < \dotsb < {\beta}_{m} < \delta + 2$ holds.
  For all $0 \leq i \leq m$, define ${M}_{i} = {N}_{{\beta}_{i}}$.
  Next suppose $0 \leq i \leq m - 1$.
  Then ${\eta}_{i + 1} = {\gamma}_{j}$ for some $j \in M$.
  Define ${\alpha}_{i} = {\alpha}_{{\gamma}_{j}}$.
  It is easy to check that $\{{\alpha}_{0}, \dotsc, {\alpha}_{m - 1}\} \in {X}^{\[m\]}$.
  So Lemma \ref{lem:goingdown} applies (with $m - 1$ as $n$) and implies that there exists ${\alpha}_{{\gamma}_{n}} \in {M}_{0} = {N}_{{\beta}_{0}} = {N}_{\left( {\gamma}_{n} \right) + 1}$ such that ${\alpha}_{{\gamma}_{n}} \in X$ and $\{{\alpha}_{0}, \dotsc, {\alpha}_{m - 1}\} \in {\left( {K}_{1}({\alpha}_{{\gamma}_{n}}) \right)}^{\[m\]}$.
  This ${\alpha}_{{\gamma}_{n}}$ is as needed.
  This completes the construction and the proof.
 \end{proof}
 \begin{Cor} \label{cor:maincor2}
  Let $\kappa > \omega$ be a measurable cardinal.
  There is a c.\@c.\@c.\@ forcing extension in which $\c = \kappa$, $\MA(\sigma-\textrm{linked})$ holds, and $\c \rightarrow {\left( \c, \alpha \right)}^{2}$, for all $\alpha < {\omega}_{1}$.
  In particular in this extension, $\add(\N) = \b = \kappa = \c$ and $\kappa \rightarrow {\left(\kappa, \alpha \right)}^{2}$, for all $\alpha < {\omega}_{1}$.
 \end{Cor}
 \begin{proof}
  By standard arguments, ${\MA}_{< \kappa}(\sigma-\textrm{linked})$ is equivalent to ${\MA}_{< \kappa}(\sigma-\textrm{linked})$ restricted to posets of size less than $\kappa$.
  Using the facts that $\kappa$ is regular, that $\kappa > \omega$, and that ${2}^{< \kappa} = \kappa$, it is possible to fix a bookkeeping function $F$ which will ensure that all names for posets of size less than $\kappa$ are eventually considered during our iteration.
  We do an FS iteration $\langle {\P}_{\alpha}, {\mathring{\Q}}_{\alpha}: \alpha \leq \kappa \rangle$ as follows.
  At a stage $\alpha < \kappa$, let $F(\alpha)$ be the object handed to us by the bookkeeping function.
  If $F(\alpha)$ is a ${\P}_{\alpha}$-name and ${\forces}_{\alpha} ``F(\alpha) \ \text{is a} \ \sigma-\textrm{linked} \ \text{poset such that} \ \lc F(\alpha) \rc < \kappa''$, then let ${\mathring{\Q}}_{\alpha}$ be a full ${\P}_{\alpha}$-name such that ${\forces}_{\alpha} F(\alpha) = {\mathring{\Q}}_{\alpha}$.
  Otherwise we let ${\mathring{\Q}}_{\alpha}$ be a full ${\P}_{\alpha}$-name for the trivial poset.
  This specifies the iteration.
  This is an FS iteration of c.\@c.\@c.\@ posets, and hence ${\P}_{\kappa}$ is c.\@c.\@c.\@
  Simple cardinality calculations show that $\lc {\P}_{\kappa} \rc = \kappa$.
  It follows that if $G$ is any $(\V, {\P}_{\kappa})$-generic filter, then in $\VG$, $\c \leq \kappa$.
  The bookkeeping function $F$ ensures that ${\MA}_{< \kappa}(\sigma-\textrm{linked})$ holds in $\VG$.
  Hence $\c = \kappa$ and $\MA(\sigma-\textrm{linked})$ holds.
  This implies that $\add(\N) = \b = \c = \kappa$.
  Finally by Lemmas \ref{lem:normal}, \ref{lem:sigmaboundedquotients}, and \ref{lem:knast} there is a normal ideal $\II$ on $\kappa$ such that for each $1 \leq n < \omega$, ${\left( {\II}^{n} \right)}^{+}$ has the $\kappa$-Knaster property.
  Therefore Theorem \ref{thm:main3} applies and implies that $\kappa \rightarrow {\left( \kappa, \alpha \right)}^{2}$ holds in $\VG$ for all $\alpha < {\omega}_{1}$.
 \end{proof}
 Note that $\b$ is equal to a former large cardinal in our model, meaning that, among other things, $\b$ is a fixed point of the $\aleph$--operation.
 So we can ask about the relation $\b \rightarrow {\left( \b, \alpha \right)}^{2}$ for smaller more accessible values of $\b$.
 This is of course related to the issue of whether any large cardinals are needed to show the consistency of $\b \rightarrow {\left( \b, \alpha \right)}^{2}$, for all $\alpha < {\omega}_{1}$.  
 \begin{Question} \label{q:partition6}
 What is the consistency strength of the relation $\b \rightarrow {\left( \b, \alpha \right)}^{2}$, for all $\alpha < {\omega}_{1}$?
 Can it be proved from a weakly compact cardinal?
\end{Question}
\begin{Question} \label{q:partition7}
 Is it consistent to have $\c = \b = {\omega}_{2}$ and $\c \rightarrow {(\c, \alpha)}^{2}$ for all $\alpha < {\omega}_{1}$?
\end{Question}
Note that a positive answer to Question \ref{q:partition7} requires that there be no ${\omega}_{2}$-Suslin trees.
Finally, we can ask whether the $\MA(\sigma-\textrm{linked})$ in Corollary \ref{cor:maincor2} can be replaced with a stronger forcing axiom.
\begin{Question} \label{q:partition8}
 Is $\ZFC + \MA + \forall \alpha < {\omega}_{1}\[\c \rightarrow {\left( \c, \alpha \right)}^{2}\]$ consistent?
\end{Question}
According to Theorem 2 of \cite{realpartition}, there is a c.\@c.\@c.\@ poset of size $\kappa$ which forces $\kappa \not\rightarrow {\left( \kappa, \omega + 2\right)}^{2}$ as long as $\cf(\kappa) > \omega$.
This example shows that even when $\kappa$ is a measurable cardinal, there need not be a normal ideal on $\kappa$ satisfying the hypothesis of Theorem \ref{thm:main3} after an FS iteration of c.\@c.\@c.\@ posets each having size less than $\kappa$ is performed.
\def\polhk#1{\setbox0=\hbox{#1}{\ooalign{\hidewidth
  \lower1.5ex\hbox{`}\hidewidth\crcr\unhbox0}}}
\providecommand{\bysame}{\leavevmode\hbox to3em{\hrulefill}\thinspace}
\providecommand{\MR}{\relax\ifhmode\unskip\space\fi MR }
\providecommand{\MRhref}[2]{%
  \href{http://www.ams.org/mathscinet-getitem?mr=#1}{#2}
}
\providecommand{\href}[2]{#2}

\end{document}